\documentclass[11pt]{amsart}
\usepackage[margin=0.85in]{geometry}
\usepackage{graphicx}
\usepackage{amssymb}
\usepackage{epstopdf}
\usepackage{xcolor}
\usepackage{mathabx}

\usepackage{hyperref}
\usepackage{cleveref}

\newcommand{\comment}[1]{}

\newcommand{\crb}{\mathcal{V}}

\newcommand{\C}{\mathbb{C}}

\newcommand{\Rn}{\mathbb{R}^n}

\newcommand{\R}{\mathbb{R}}

\newcommand{\N}{\mathbb{N}}

\newcommand{\dop}[1]{\frac{\partial}{\partial #1}}

\newcommand{\dopt}[2]{\frac{\partial #1}{\partial #2}}

\newcommand{\fpstwo}[2]{#1[[#2]]}

\newcommand{\todo}[1]{}

\newlength{\extendaxesby}
\newcommand{\cinfty}{\mathcal{C}^\infty}
\newcommand{\diffable}[1]{\mathcal{C}^{#1}}

\newcommand{\ssheaf}{\mathcal{S}}
\newcommand{\esheaf}{\mathcal{E}}
\newcommand{\holforms}{\mathrm{T}'}
\newcommand{\charforms}{\mathrm{T}^0}
\newcommand{\holsheaf}{\mathcal{T}'}
\newcommand{\charsheaf}{\mathcal{T}^0}
\newcommand{\vectorfields}{\mathfrak{X}}
\newcommand{\distributions}{\mathcal{D}'}
\DeclareMathOperator{\singsupp}{sing supp}
\DeclareMathOperator{\ann}{ann}

\DeclareMathOperator{\rank}{rk}

\DeclareMathOperator{\imag}{Im}
\DeclareMathOperator{\real}{Re}

\DeclareMathOperator{\Hom}{Hom}
\DeclareMathOperator{\ourL}{L}
\DeclareMathOperator{\ourd}{d}

\renewcommand{\epsilon}{\varepsilon}
\renewcommand{\Re}{\mathrm{Re}\,}
\renewcommand{\Im}{\mathrm{Im}\,}

\theoremstyle{definition}
\newtheorem{defn}{Definition}[section]

\theoremstyle{remark}
\newtheorem{exa}[defn]{Example}
\newtheorem{rmk}[defn]{Remark}

\theoremstyle{plain}
\newtheorem{prop}[defn]{Proposition}
\newtheorem{cor}[defn]{Corollary}
\newtheorem{thm}[defn]{Theorem}
\newtheorem{lem}[defn]{Lemma}
\usepackage{paralist}
\usepackage{tikz}

\usepackage{faktor}

\title[Regularity of infinitesimal automorphisms of involutive structures]{Regularity of infinitesimal automorphisms \\of involutive structures}
\author{Bernhard Lamel}\email{bernhard.lamel@univie.ac.at}
\author{Nicholas Braun Rodrigues}
\begin{document}
\maketitle

\begin{abstract}
  In this paper, we prove that infinitesimal automorphisms of an involutive structure are
  smooth. For this, we build a regularity theory
  for sections of vector bundles over an involutive structure $(M,\crb)$ endowed
  with a connection compatible with $\crb$, which we call $\crb$-connection. We show that
  $\crb$-sections, i.e. sections which are parallel with respect to $\crb$ under
  the $\crb$-connection, satisfy an analogue of Hans Lewy's theorem as formulated
  for CR functions on an abstract CR manifold by Berhanu and Xiao, and introduce
  certain (generically satisfied) nondegeneracy conditions ensuring their smoothness. 
\end{abstract}

%%%%%%%%%%%%%%%%%%%%%%%%%%%%%%%
%%%%%%%%%%%%%%%%%%%%%%%%%%%%%%%
%
\section{Introduction}\label{sec:intro}
%
%%%%%%%%%%%%%%%%%%%%%%%%%%%%%%%
%%%%%%%%%%%%%%%%%%%%%%%%%%%%%%%

The purpose of this paper is twofold: First, we study microlocal regularity of certain vector-valued analogues
of solutions of involutive structures and second, we apply this to the study of regularity
of infinitesimal automorphisms of involutive structures. 

Historically, our first topic reaches back to Hans Lewy's
discovery of ``atypical'' partial differential equations and the related
regularity properties in \cite{lewy56}. Atypical refers to the fact
that those equations do not possess any solutions, and Lewy connected this
to the fact that solutions of the associated homogeneous equation can
be interpreted as boundary values of {\em holomorphic} functions, thus being
far from arbitrary. The close relationship between PDE and several complex variables
thus had one more rich facet, giving rise to the theory of involutive and CR structures
in particular, and one of the longstanding problems in CR geometry is to understand the microlocal
regularity of CR functions.

For embedded (i.e. locally integrable) CR manifolds, many authors
have pursued the problem of holomorphic extendability (or microlocal
hypo-analyticity) of CR functions. Lewy's result can be rephrased in
more general terms by saying that the {\em hypo-analytic wave front set} of
a CR function cannot include any non-zero eigenvectors of positive (or negative,
depending on convention) eigenvalues of the Levi form. It turns out that holomorphic
extendability to a wedge is characterized by an even weaker condition, namely
{\em minimality} of a CR structure, by work of Tumanov \cite{tumanov1989} and Baouendi and Rothschild \cite{baouendi1988},
based on the method of analytic discs. Another approach, introduced by  Baouendi, Chang and Treves \cite{baouendi1983} was based
on a variant of the FBI transform.

When dealing with abstract CR or more generally involutive structures, much less is known.
Berhanu and Xiao proved the abstract CR analogue of Hans Lewy's theorem in \cite[Theorem 2.9]{berhanu2015}
and applied it to the study of regularity of mappings from an abstract CR structure to a strictly
pseudoconvex hypersurface; embeddability of the target is a key property utilized in the proof,
as a CR map consists of CR functions (i.e. solutions) in this setting. If one studies
{\em abstract} CR structures, one can prove regularity of infinitesimal automorpisms
of nondegenerate CR structures with very low regularity
assumptions, such as in F\"urd\"os and Lamel \cite{Furdos:2016vu} if they
have the ``microlocal extension property'', as defined in a paper of Lamel and Mir \cite{MR4713116}.
In this approach, one interprets infinitesimal automorphisms as sections of appropriate
vector bundles which are parallel with respect to a canonical partial connection. Our first
result bridges the gap left in these papers by the absence of a vector-valued
analogue of Hans Lewy's theorem and gives a simple condition which ensures the microlocal
extension property for sections which are solutions (we refer to these as $\crb$-sections) holds. 

\begin{thm}\label{thm:levi_microlocal_regularity}
  Let $\Omega$ be a $\mathcal{C}^\infty$-smooth manifold endowed with
  an involutive structure $\mathcal{V}$, and let $E$ be a
  $\mathcal{V}$-bundle over $\Omega$. Let $p \in \Omega$ and
  $\sigma \in \mathrm{T}^0_p$, and assume that the Levi form in
  $\sigma$ has one negative eigenvalue. Then for any 
  distributional $\crb$-section  $u$ defined on some open neighborhood of $p$,
  we have $\sigma \notin WF(u)$.
\end{thm}

We note here the important corollary in the CR case:

\begin{cor} Assume that $M$ is an abstract Levi-nonflat CR manifold. Then $M$ has
  the microlocal extension property for sections of CR bundles over $M$. In particular,
  if $M$ is in addition finitely nondegenerate, then every
  bounded infinitesimal automorphism of $M$ is smooth. 
\end{cor}

 The novelty here is that instead of a homogeneous, purely
first order partial differential equation we have a first order term
that involves all the components of the section, \textit{i.e.} the
equation that a $\crb$-section $u$ satisfies (locally) is of the kind
\[
  L u + A u = 0,
\]
where $L$ is a section of $\crb$, and $A$ is an $r\times r$ matrix with
$\mathcal{C}^\infty$-smooth coefficients, and $r$ is the complex
dimension of the $\crb$-bundle.

Our second main topic is the regularity of infinitesimal automorphisms of involutive structures, which
seems to be a completely open problem. One of the issues that arises in the non-integrable setting
is that it becomes hard to interpret components of a map as solutions to a PDE. We show
how to overcome this problem under certain natural conditions by interpreting an infinitesimal
automorphism $X$ as a section $X^\sharp$ of the ``bidual'' $(T'M)^* \subset (\C T^*M)^*$. This comes with several hitherto
unexplored problems, as we have to impose certain conditions which imply that a vector field
is actually determined by this associated section $X^\sharp$, which is automatically true in the CR setting. The
conditions we introduce are  sufficient to ensure  that any locally bounded infinitesimal automorphism
of a smooth involutive structure satisfying these conditions is smooth, and we show by examples that
generically they are also necessary. One of them is called {\em nondegeneracy} and introduced in
Defintion~\ref{def:degeneracy}. This condition yields exactly the finite nondegeneracy condition introduced 
in the CR setting by Baouendi, Huang, and Rothschild \cite{baouendi1996}, but at the level of generality here, relies on the $\crb$-module calculus introduced in Section~\ref{sec:cr_modules}
below; we can also reduce the nondegeneracy assumption to include certain mild degeneracies, which we call ``normal crossing degeneracy locus''. The second condition we  call ``normal crossing exceptional locus''
and it ensures that the smoothness of $X^\sharp$ is actually related to the smoothness of $X$. It
is introduced in  Definition~\ref{def:normalcrossings} below. This condition is automatically
satisfied by CR structures. 

We note that both of these conditions are generically {\em necessary} in the sense that
if they are not fulfilled on an open subset, we give examples of nonsmooth infinitesimal
automorphisms. 

\begin{thm}\label{thm:mainautos}
  Assume that $(M,\crb)$ is an involutive structure, $p\in M$. If $M$ is 
  nondegenerate at $p$ and possesses normal crossing exceptional locus,
  then every bounded infinitesimal automorphism $X$ near $p$ which extends
  microlocally is smooth near $p$. 
\end{thm}

Combining Theorem~\ref{thm:levi_microlocal_regularity} and Theorem~\ref{thm:mainautos} yields the following powerful
corollary.
\begin{cor}\label{cor:mainautos}
  Assume that $(M,\crb)$ is a Levi-nonflat involutive structure, $p\in M$. If $M$ is 
  nondegenerate at $p$ and possesses normal crossing exceptional locus,
  then every bounded infinitesimal automorphism $X$ near $p$  is smooth near $p$. 
\end{cor}

Let us again outline the main difficulties we encounter in the involutive setting. 
One great
advantage in the CR case is that one
studies hypo-analytic (microlocal) regularity (or holomorphic
extendability) which turns out to reduce to the study of  the trace of CR functions
to maximally real submanifolds, because of the Baouendi-Treves
approximation formula. The same type of  argument fails for
$\mathcal{C}^\infty$-smooth (microlocal) regularity, and this might be
the source of the difficulty for obtaining microlocal regularity
results for CR functions on the abstract setting. Also, in the
abstract case, there is no notion of holomorphic extendability, so we
do not benefit from the theory of holomorphic functions anymore.
Therefore, moving from  the CR case to the one of general
involutive structures (see section \ref{sec:involutive} for more
details) is a rather big step.

%%%%%%%%%%%%%%%%%%%%%%%%%%%%%%%
%%%%%%%%%%%%%%%%%%%%%%%%%%%%%%%
%
\section{Involutive structures}\label{sec:involutive}
%
%%%%%%%%%%%%%%%%%%%%%%%%%%%%%%%
%%%%%%%%%%%%%%%%%%%%%%%%%%%%%%%
In this section we shall recall some basic definitions and facts about
involutive structures, and we refer to Chapter I of \cite{cordarobook}
for more details.

Let $M$ be a $\mathcal{C}^\infty$-smooth, $n+m$-dimensional
manifold. We say that a subbundle
$\mathcal{V} \subset \mathbb{C} TM$ defines an \emph{involutive
  structure} on $M$ if
$[\mathcal{V},\mathcal{V}] \subset \mathcal{V}$, \textit{i.e.} for
every open set $U \subset M$, if
$L_1,L_2 \in \Gamma(U, \mathcal{V})$ then
$[L_1, L_1] \in \Gamma(U, \mathcal{V})$. For a given involutive
structure $\mathcal{V} \subset \mathbb{C}T M$, we define the
bundle $\mathrm{T}^\prime$ (or $\mathcal{V}^\perp)$ as the subbundle
of $\mathbb{C}T^\ast M$ given by the one-forms that annihilate
every section of $\mathcal{V}$, \textit{i.e.} for every open set
$U \subset M$, a one-form
$\omega \in \Gamma(U, \mathbb{C} T^\ast M)$ is a section of
$\mathrm{T}^\prime$ if and only if
\begin{equation*}
  \omega(L) = 0, \qquad \forall L \in \Gamma(U, \mathcal{V}).
\end{equation*}
We set the characteristic set of $\mathcal{V}$ as
$\mathrm{T}^0 = \mathrm{T}^\prime \cap T^\ast M$, and write $\holforms(M)$
resp. $\charforms{M}$ if we need to refer to the structure. We note that in
general $\mathrm{T}^0$ is not a bundle, but at least
$\dim \mathrm{T}^0$ is upper semicontinuous. We use the convention
that $m=\dim \mathcal{V}$ and $n= \dim T'$. A number of involutive
structures are of special importance, and have been given names; we
say that $\mathcal{V}$ is:
\begin{itemize}
\item \emph{elliptic} if $\mathrm{T}^0_p = 0$ for every
  $p \in M$;
\item \emph{complex} if
  $\mathcal{V}_p \oplus \overline{\mathcal{V}}_p =
  \mathbb{C}T_pM$ for every $p \in M$;
\item \emph{CR} if $\mathcal{V}_p \cap \overline{\mathcal{V}}_p = 0$
  for every $p \in M$;
\item \emph{real} if $\mathcal{V}_p =\overline{\mathcal{V}}_p$ for
  every $p \in M$.
\end{itemize}
A function (or distribution) $u$ defined in some open set
$U \subset M$ is said to be a \emph{solution} if $Lu = 0$ for
every $L \in \Gamma(U,\mathcal{V})$ (if $\mathcal{V}$ is CR then we
call $u$ a CR function). Let $p \in M$, and suppose that
$\mathrm{dim} \mathrm{T}^0_p = d$ and take $\theta_1, \dots, \theta_d$ real forms in an open neighborhood of $p$ such that $\{ \theta_1 (p), \dots, \theta_d (p)\}$ span $\mathrm{T}^0_p$.

%\footnote{I have changed that - on a neigbourhood can't be, right?}
Then
there exists a local coordinate system vanishing at $p$,
$(U, x_1, \dots, x_\nu, y_1, \dots, y_\nu, s_1, \dots, s_d, t_1,
\dots, t_\mu)$, and one-forms
$\omega_1, \dots, \omega_\nu$ such that,
writing $z_j = x_j + iy_j$ for $j =1, \dots, \nu$, we have that
$\{\omega_1, \dots, \omega_\nu, \theta_1, \dots, \theta_d\}$ spanns
$\Gamma(U, \mathrm{T}^\prime)$, and
\begin{equation*}
  \begin{cases}
    \mathrm{d}z_j\big|_p = \omega_j \big|_p, \quad j = 1, \dots, \nu;\\
    \mathrm{d}s_k \big|_p = \theta_k\big|_p, \quad k = 1, \dots, d.
  \end{cases}
\end{equation*}
Furthermore, there exists complex vector fields
$L_1, \dots, L_{\nu + \mu}$ spanning $\Gamma(U, \mathcal{V})$ of the
form
\begin{align*}
  &L_{j} = \frac{\partial}{\partial \bar{z}_j} + \sum_{\ell = 1}^\nu a_{j\ell} \frac{\partial}{ \partial z_\ell} + \sum_{k =1}^d b_{jk}\frac{\partial}{\partial s_k}, \qquad j = 1, \dots, \nu;\\
  &L_{j} = \frac{\partial}{\partial t_{j-\nu}} + \sum_{\ell = 1}^\nu a_{j\ell} \frac{\partial}{ \partial z_\ell} + \sum_{k =1}^d b_{jk}\frac{\partial}{\partial s_k}, \qquad j = \nu+1, \dots, \nu+\mu,
\end{align*}
where the coefficients $a_{j\ell}$ and $b_{jk}$ are smooth on $U$ and
they vanish at $p$. If $\mathcal{V}$ is CR, then we can take
$\mu = 0$, and if $\mathcal{V}$ is elliptic, $d = 0$. An important
Hermitian form on $\mathcal{V}$ is the so-called the \emph{Levi form}:
For $p \in M$ and $\xi \in \mathrm{T}^0_p$, $\xi \neq 0$, the
Levi form of $\mathcal{V}$ in $\xi$ is the Hermitian form given by
\begin{equation*}
  \mathfrak{L}^\xi_p(v,w) =
  \frac{1}{2i}\xi   \left(\left[\,  L, \overline{K}\, \right]\right), \quad \forall v,w \in \mathcal{V}_p,
\end{equation*}
where $L$ and $K$ are smooth sections of $\mathcal{V}$ on an open
neighborhood of $p$ such that $L_p = v$ and $K_p = w$.
An important class of involutive structures are the so called
\emph{locally integrable structures}. We say that an involutive
structure $\mathcal{V}$ is locally integrable (and call it a locally
integrable structure of rank $n$) on $M$ if, for every
$p \in M$, there exists an open neighborhood $U$ of $p$ and
complex-valued smooth functions $Z_1, \dots, Z_m$ (sometimes called
\emph{first integrals}) defined on $U$, such that
\begin{equation*}
  \begin{cases}
    \mathrm{d}Z_1 \wedge \cdots \wedge \mathrm{d} Z_m \neq 0, \, \text{on}\, U;\\
    \mathrm{d}Z_j (L) = 0,\quad \forall L \in \Gamma(U, \mathcal{V}), \, j = 1, \dots, m.
  \end{cases}
\end{equation*}
In this case we have a better local description of the sections of
$\mathcal{V}$. Let $p \in M$ be arbitrary, and let
$d = \text{dim} \mathrm{T}^0_p$, then there exist a local coordinate
system $(U, x_1, \dots, x_\nu, y_1, \dots, y_\nu, s_1, \dots, s_d, t_1,
\dots, t_\mu)$ vanishing at $p$ and a map $\phi : U \longrightarrow \mathbb{R}^m$,
with $\phi(0) = 0$ and $\mathrm{d}\phi(0) = 0$, such that the
differentials of the functions
\begin{align*}
  &Z_k(x,y,s,t) = x_k + iy_k, \quad k = 1, \dots, \nu;\\
  &W_k(x,y,s,t) = s_k + i\phi_k (x,y,s,t), \quad k =1, \dots, d,
\end{align*}
span $\mathrm{T}^\prime$ on $U$. Thus the following complex vector
fields span $\mathcal{V}$ on $U$:
\begin{align*}
  &L_j = \frac{\partial}{\partial \bar{z}_j} - i \sum_{k=1}^d \frac{\partial \phi_k}{\partial \bar{z}_j} N_k, \quad  j = 1, \dots, \nu;\\
  &L_{\nu+ j} = \frac{\partial}{\partial t_j} - i\sum_{k=1}^d \frac{\partial \phi_k}{\partial t_j} N_k, \quad j=1, \dots, \mu,
\end{align*}
where the complex vector fields $N_1, \dots, N_d$ are giving by
\begin{equation*}
  N_k = \sum_{\ell = 1}^d \big({}^tW_s^{-1}\big)_{k\ell} \frac{\partial}{\partial s_\ell}, \quad k = 1, \dots, d,
\end{equation*}
and we are assuming $U$ small enough so $W_s$ is invertible on $U$. By
Newlander-Nirenberg Theorem we have that every complex involutive
structure is locally integrable, and as a consequence, every elliptic
involutive structure is also locally integrable. We also note that if
$M$ is $\mathcal{C}^\omega$-smooth and if $\mathcal{V}$ is
$\mathcal{C}^\omega$-smooth, then $\mathcal{V}$ is locally integrable.
\begin{exa}
  Let $M \subset \mathbb{C}^{n+1}$ be a real hypersurface given by
  \begin{equation*}
    M = \{(z,w) \in \mathbb{C}^n \times \mathbb{C} \,: \, \Im w = \phi(z,\bar z, \Re w) \},
  \end{equation*}
  for some real-valued, smooth function $\phi$, such that
  $\phi(0) = 0$. We can write $M$ as the image of the map
  $\mathbb{R}^{2n+1} \ni (x,y,s) \mapsto (x + iy, s + i\phi(x,y,s))$,
  which then defines a locally integrable structure on $M$ (which is
  CR).
\end{exa}
If $\mathcal{V}$ is a CR involutive structure, we call the pair
$(M, \mathcal{V})$ an abstract CR manifold, and if $\mathcal{V}$
is locally integrable, then we say that $(M, \mathcal{V})$ is
locally embeddable. It
is also worthwhile noticing that in the CR case (even the abstract
one), the characteristic set $\mathrm{T}^0$ is actually a bundle.

An important class of locally integrable structures that are not CR
are the so-called Mizohata structures. An involutive structure
$\mathcal{V}$ on some $\mathcal{C}^\infty$-smooth, $n+1$-dimensional
manifold $M$ is called a Mizohata structure, if it satisfies:
\begin{itemize}
\item $\mathcal{V}$ has rank $n$;
\item the characteristic set $\mathrm{T}^0$ is not trivial;
\item the Levi form is nondegenerate in every point of
  $\mathrm{T}^0\setminus \{0\}$.
\end{itemize}
\begin{exa}
  (The standard Mizohata structure) Let $0 \leq \nu \leq n$. The
  standard Mizohata structure of type $\{\nu, n- \nu\}$ on
  $\mathbb{R}^{n+1}$ is the locally integrable structure generated by
  the following complex vector fields:
  \begin{equation*}
    L = \frac{\partial}{\partial t_j} - i\epsilon_jt_j\frac{\partial}{\partial x}, \quad t = 1, \dots, n,
  \end{equation*}
  where $\epsilon_j = 1$ for $j = 1, \dots, \nu$, and
  $\epsilon_j = -1$ for $j = \nu + 1, \dots, n$. So the characteristic
  set is given by
  \begin{equation*}
    \mathrm{T}^0_{x,t} = 
    \begin{cases}
      \{\lambda \mathrm{d}x \, : \, \lambda \in \mathbb{R} \}, &  t = 0,\\
      \{0\}, & t \neq 0,
    \end{cases}
  \end{equation*}
  and the projection of $\mathrm{T}^0$ to the base manifold is the
  curve $\Sigma = \{(x,t) \in \mathbb{R}^{n+1} \, : \, t = 0\}$.
\end{exa}
Actually, every Mizohata structure can be modeled on the standard one
in the following manner. Let $\mathcal{V}$ be a Mizohata structure on
$M$, and denote by $\Sigma$ the projection of $\mathrm{T}^0$ on
the base manifold $M$. Now let $p \in \Sigma$, so
$\text{dim}\mathrm{T}^0 = 1$. Suppose that the Levi form in the
characteristic direction at $p$ has $\nu$ positive and $n - \nu$
negative eigenvalues. Therefore (see for instance Theorem VII.1.1 of
\cite{trevesbook}) one can find a local coordinate system vanishing at
$p$, $(U, x, t_1, \dots, t_n)$, such that $\mathcal{V}$ is generated
on $U$ by the following complex vector fields,
\begin{equation*}
  L = \frac{\partial}{\partial t_j} - i\epsilon_jt_j(1+\rho_j(x,t))\frac{\partial}{\partial x}, \quad t = 1, \dots, n,
\end{equation*}
where $\epsilon_j = 1$ for $j = 1, \dots, \nu$, and $\epsilon_j = -1$
for $j = \nu_1, \dots, n$, and the functions $\rho_j$ are smooth and
vanish of infinite order at $t=0$ for every $j$. Furthermore, the
Mizohata structure is integrable if and only if we can take
$\rho \equiv 0$ (see Theorem VII.1.1 of \cite{trevesbook}). This
Mizohata structures provide us with a large set of examples of
involutive structures outside the "CR world", and they were
extensively studied, see for instance
\cite{hounie1992,hounie1993,hounie1997,meziani1995,meziani1997,meziani1999,jacobowitz1993,sjostrand1980}
.
%%%%%%%%%%%%%%%%%%%%
\subsection{Compatible maps and infinitesimal $\mathcal{V}$-automorphisms}

\begin{defn}
  \label{def:compatible} If $M$ and $M'$ are endowed with involutive structures $\crb$ and $\crb'$,
  respectively, then we say that a map $h\colon M\to M'$ of class at least $\diffable{1}$ is
  {\em compatible} if $h_* \crb \subset \crb'$. In this case, we write $h\colon (M,\crb) \to (M',\crb')$. 
\end{defn}

\begin{lem}
  \label{lem:compatible} Let $h\colon M\to M'$ be a $\diffable{1}$-smooth  map. Then $h$ is
  compatible if and only if $h^* \holforms{M'} \subset \holforms(M)$, and we also have
  that for every solution $f$ of $\crb'$, the
  function $h^*f = f\circ h$ is a solution of $\crb$. If in addition
  $\crb'$ is locally integrable,  the following statements are equivalent:
  \begin{compactenum}
  \item $h$ is compatible;
    \item $h^*f = f\circ h$ is a solution of $\crb$ if $f$ is a solution of $\crb'$. 
  \end{compactenum}
\end{lem}

\begin{proof}
  The first statement follows easily from $(h^* \omega)_p (L_p) = \omega_{h(p)} ((h_*)_p L_p)$: If we
  assume $h$ is compatible, then the right hand side is zero for every $\omega_{h(p)} \in \mathrm{T}'_{h(p)}( M')$,
  so $(h^* \omega)_p \in \holforms(M)$; the converse follows in the same way.

  Now assume that $\crb'$ is locally integrable. If $h$ is compatible, then
  we already know that $f\circ h$ is a solution for every solution $f$. If on the
  other hand, $f \circ h$ is a solution, then $ L f\circ h= d(h^*f) (L) = (h^* df) (L) = 0$ for every section
  $L$ of $\crb$, so
  that $h^* df \in  \holforms(M)$; since the differentials of solutions span $\holforms{M'}$ by local integrability,
  it follows that $h^* \holforms{M'} \subset \holforms(M)$. 
\end{proof}

The ``tangent space'' to the space of compatible diffeomorphisms (of class at least $\diffable{1}$) consists of
vector fields satisfying a certain PDE, which we now describe (the reader can skip to
\cref{def:infaut} to see the result). Let $h_\tau : M \longrightarrow M$ be an one-parameter family of compatible maps such that $h_0 = \text{Id}$, and $\frac{\mathrm{d}}{\mathrm{d} \tau} \big|_{\tau = 0} h_\tau = X$. Let $(U, y_1, \dots, y_{N})$ be a local coordinate system on $M$, with $N = n+m$. Let  $\omega \in \Gamma(U, \mathrm{T}^\prime)$ and $L \in \Gamma(U, \mathcal{V})$. Since $h_\tau$ is a compatible map, we have that $h_\tau^*\omega(L) = 0$, for every $\tau$. Writing $\omega = \sum_{j = 1}^N \beta_j \mathrm{d} y_j$, $L = \sum_{j = 1}^N a_j \frac{\partial}{\partial y_j}$, and $h_\tau(y) = (h_\tau^1, \dots, h_\tau^N)$, we have that
\begin{align*}
	0 & = h_\tau^\ast\omega(L)(y)) \\
	& =  \omega((h_\tau)_\ast L)(h_\tau(y)) \\
	& = \sum_{k = 1}^N \beta_k(h_\tau(y)) \mathrm{d}y_k \left(\sum_{\ell,j = 1}^N a_\ell(y) \frac{h_\tau^j}{\partial y_\ell}(y) \frac{\partial}{\partial y_j} \right) \\
	& = \sum_{k, \ell = 1}^N \beta_k(h_\tau(y)) a_\ell(y) \frac{\partial h_\tau^k}{\partial y_\ell}(y) \\
	& = \sum_{k = 1}^N \beta_k(h_\tau(y)) Lh_\tau^k(y),
\end{align*}
for every $y \in U$, and every $\tau$. Now taking the $\tau$-derivative of the above equation and placing $\tau = 0$, and keeping in mind that $h_0(y) = y$ and writing $X = \sum_{j = 1}^N \alpha_j \frac{\partial}{\partial y_j}$, we obtain that
\begin{align*}
	0 & = \sum_{k,j = 1}^N \frac{\beta_k}{\partial y_j}(y)\frac{\mathrm{d}}{\mathrm{d} \tau}\Big|_{\tau = 0} h_\tau^j (y) L h_0^k(y) + \sum_{k = 1}^N \beta_k(y) L \left( \frac{\mathrm{d}}{\mathrm{d} \tau}\Big|_{\tau = 0} h_\tau^j \right) (y) \\
	& = \sum_{k = 1}^N \frac{\beta_k}{\partial y_k}(y)\alpha_k(y) a_k(y) + \sum_{k = 1}^N \beta_k(y) L \alpha_k (y) \\
	& =  \sum_{k = 1}^N X \beta_k (y) a_k(y) + \sum_{k = 1}^N \beta_k(y) L \alpha_k (y) \\
	& = - \sum_{k = 1}^N  \beta_k (y) X a_k(y) + \sum_{k = 1}^N \beta_k(y) L \alpha_k (y),
\end{align*}
for $\sum_{k = 1}^N \beta_k(y) a_k(y) = 0$. Therefore, $X$ satisfies the following equation: $\omega \big( [L, X] \big) = 0$. Since $\omega(L) = 0$, we have that $\mathrm{d}\omega(L, X) = L \omega(X) - \omega([L,X]) = L \omega(X)$, and also that $\mathcal{L}_L \omega(X) = \mathrm{d}\omega(L,X) + X \omega(L) = \mathrm{d} \omega(L,X) = L \omega(X)$, where $\mathcal{L}$ is the Lie derivative. 
\begin{defn}\label{def:infaut}
Let $X$ be a real vector field in $M$. We say that $X$ is a \emph{infinitesimal $\mathcal{V}$-automorphism} if $\omega([X,L]) = 0$ for every $\omega \in \mathrm{T}^\prime(M)$ and $L \in \mathcal{V}$. 
\end{defn}
We point out that if $\mathcal{V}$ is locally integrable, we can get a simpler equation. If $\mathrm{T}^\prime M$ is spanned on $U$ by the first integrals $Z_1, \dots, Z_m$ and then $X$ is an infinitesimal $\mathcal{V}$-automorphism if and only if $L \mathrm{d}Z_j (X) = 0$, for every $j = 1, \dots, m$. In spirit, this
corresponds to writing $X = \sum_j X_j \dop{Z_j} $ with $X_j$ being a solution. 

\subsection{An interesting example}\label{sub:example}

Here we describe an example of a locally integrable structure, which is generically strictly
pseudoconvex, but with a nontrivial exceptional locus.  
Let $k , \ell \in \mathbb{N}$. Consider in $\mathbb{R}^3$ the locally integrable structure generated by the following first integrals:
\begin{align*}
	Z(x,y,t) & = x + i \frac{t^{\ell+ 1}}{\ell + 1} \\
	W(x,y,t) & = y + i \frac{t^{k + 1}}{k + 1}.  
\end{align*}
Then $\mathcal{V}$ is spanned by the vector field 
\[
	L = \frac{\partial}{\partial t} - i t^\ell \frac{\partial}{\partial x} - i t^k \frac{\partial}{\partial y}.
\]
Note that $\mathbb{C} T \mathbb{R}^3$ is spanned by the vector fields $\{L, \partial / \partial x, \partial / \partial y\}$, and that  $\mathbb{C} T^* \mathbb{R}^3$ by the one-forms $\{\mathrm{d}t, \mathrm{d} Z, \mathrm{d} W \}$.  Consider the following real one-form $\theta \in \mathrm{T}^0 \mathbb{R}^3$ given by
\[
	\theta = t^k \mathrm{d} Z - t^\ell \mathrm{d} W. 
\]
Now observe that,
\[
	\mathcal{L}_L \theta  =  \left( \frac{\partial}{\partial t} t^\ell \right) \mathrm{d} Z - \left( \frac{\partial}{\partial t} t^k \right) \mathrm{d} W,
\]
thus,
\[
	\mathcal{L}_{L^q} \theta = \mathcal{L}_L \cdots \mathcal{L}_L \theta =   \left( \frac{\partial^q}{\partial t^q} t^\ell \right) \mathrm{d} Z - \left( \frac{\partial^q}{\partial t^q} t^k \right) \mathrm{d} W =  \left(q! \binom{\ell}{q} t^{\ell-q} \right) \mathrm{d} Z - \left( k! \binom{\ell}{k} t^{k-q} \right) \mathrm{d} W.
\]
Therefore, if $k \neq \ell$, we have that $\langle \mathcal{L}_{L^k} \theta, \mathcal{L}_{L^\ell} \theta \rangle = \mathrm{T}^\prime \mathbb{R}^3$. Now let $X = a \partial / \partial t + b \partial / \partial x + c \partial / \partial y$ be a real vector field. Then $X$ is an infinitesimal $\mathcal{V}$-automorphism if
\begin{align*}
	\left(\frac{\partial}{\partial t} - i t^\ell \frac{\partial}{\partial x} - i t^k \frac{\partial}{\partial y}\right)\left(i t^\ell a + b \right) & = 0, \\
	\left(\frac{\partial}{\partial t} - i t^\ell \frac{\partial}{\partial x} - i t^k \frac{\partial}{\partial y}\right)\left(i t^k a + c \right) & = 0,
\end{align*}
in other words,
\begin{align*}
	& \frac{\partial b}{\partial t}  + t^{2\ell} \frac{\partial a}{\partial x} + t^{\ell + k} \frac{\partial a}{\partial y}  = 0,\\
	 &\frac{\partial}{\partial t} (t^\ell a) - t^\ell \frac{\partial b}{\partial x} - t^k \frac{\partial b}{\partial y}  = 0, \\
	&\frac{\partial c}{\partial t}  + t^{\ell + k} \frac{\partial a}{\partial x} + t^{2 k} \frac{\partial a}{\partial y} = 0, \\
	 &\frac{\partial}{\partial t} (t^k a) - t^\ell \frac{\partial c}{\partial x} - t^k \frac{\partial c}{\partial y}  = 0 \\
\end{align*}
One can now see by hand that we can completely eliminate the
derivatives of $a$ from this PDE system: the first and third equations yield $t^k b_t = t^\ell c_t$,
and the second and forth equation yields
\[ (\ell -  k )t^{k+\ell -1} a = t^{\ell + k} b_x + t^{2 k} b_y - t^{2\ell} c_x - t^{\ell + k} c_y.   \]
The structural reason for this is going to be made clear by the interpretation of the
condition to be an infinitesimal automorphism as being parallel with respect to a certain connection. The next sections introduce the necessary technicalities.

%%%%%%%%%%%%%%%%%%%%%%%%%%%%%%%%%
%%%%%%%%%%%%%%%%%%%%%%%%%%%%%%%%%
%
\section{$\mathcal{V}$-bundles}\label{sec:V_bundles}
%
%%%%%%%%%%%%%%%%%%%%%%%%%%%%%%%%%
%%%%%%%%%%%%%%%%%%%%%%%%%%%%%%%%%
%
In this section we generalize the notion of an abstract CR
bundle using the notion of involutive structures. Let $M$ be a
$\mathcal{C}^\infty$-smooth, $n+m$-dimensional manifold endowed with a
involutive structure $\mathcal{V}\subset \mathbb{C}TM$ of rank
$n$.
\begin{defn}
  We say that a $\mathcal{C}^\infty$-smooth complex vector bundle $E$
  over $M$ is a $\mathcal{V}$-bundle over $M$ if it is
  equipped with a flat partial connection $D$, defined on
  $\mathcal{V}$.

  More precisely, for every open set $U \subset M$, $D$ satisfies
  the following properties:
  \begin{compactenum}[\rm i)]
  \item $D_{a L + K} \omega = a D_{L} \omega + D_{K} \omega$, \, for
    all $a \in \diffable{\infty}(U)$,
    $L, K \in \Gamma (U, \mathcal{V})$, $\omega\in \Gamma (U,E)$.
  \item $D_{L} (\omega + \eta) = D_{L} \omega + D_{L}\eta $, \, for
    all $L \in \Gamma (U, \mathcal{V})$,
    $\omega, \eta\in \Gamma (U,E)$.
  \item $D_{L} a \omega = (L a) \omega + a D_{L} \omega$,\, for all
    $a \in \diffable{1}(U)$, $L \in \Gamma (U, \mathcal{V})$,
    $\omega\in \Gamma (U,E)$.
  \item For all $L, K \in \Gamma(U, \mathcal{V})$
    % there exists $\bar M \in \Gamma(U, \crb)$ such that
    % $[D_{\bar L}, D_{\bar K} ] = D_{\bar M} $.
    $[D_{L}, D_{K} ] = D_{[L, K]} $.
  \end{compactenum}
\end{defn}

\begin{exa} The flat bundle $M \times \C^r$ has a canonical $\mathcal{V}$-vector bundle structure by 
declaring that constant sections are solutions; using $e^1, \dots, e^r$ as sections, 
this means that $D_{L} \eta_j e^j = (L \eta_j) e^j$.
\end{exa}

\begin{exa}
The orthogonal bundle of $\mathcal{V}$, $\mathrm{T}^\prime$, becomes a $\mathcal{V}$-bundle with the Lie derivative of forms, 
\[ (D_{L} \omega) (X) = L \omega (X) - \omega ([L, X]) = d\omega (L, X). \]
First note that the bundle $\mathrm{T}^\prime$ is a complex bundle over $M$ for it is a subbundle of $\mathbb{C}T^\ast M$, and the later has a natural complex vector bundle structure.
Since $[\mathcal{V},\mathcal{V}]\subset \mathcal{V}$, it is easy to see that $D_L\omega \in \Gamma(U,\mathrm{T}^\prime)$ for all $L \in \Gamma(U, \mathcal{V})$ and $\omega \in \Gamma(U,\mathrm{T}^\prime)$. The only non-obvious condition that we need to check is \textit{iv)}: For any $\omega \in \Gamma(U, \mathrm{T}^\prime)$, $L,K \in \Gamma(U,\mathcal{V})$, and $X \in \Gamma (U, \mathbb{C}TU)$, we have that
\begin{align*}
	[D_L,D_K]\omega (X) & = D_L(D_K\omega)(X) - D_K(D_L\omega)(X)\\
	&=   L \big(D_K\omega (X)\big) - D_K\omega ([L, X]) -  K\big( D_L\omega (X)\big) + D_L\omega ([K, X]) \\
	&= L \big( K \omega (X) - \omega ([K, X]) \big) - K\omega([L,X]) + \omega([K,[L,X]]) \\
	&\quad - K\big( L\omega(X) - \omega([L,X]) \big) + L\omega([K,X]) - \omega([L,[K,X]]) \\
	&= [L,K]\omega(X) - \omega([K,[X,L]] + [L,[K,X]]) \\
	&= [L,K]\omega(X) - \omega([[L,K],X]) \\
	&= D_{[L,K]}\omega(X).
\end{align*}
\end{exa}
%
%\[ 
%\begin{aligned}
%[D_{\bar L}, D_{\bar K}] \omega &= \left( [\bar L, \bar K ]  
%+ (\bar L D(\bar K) - \bar K D(\bar L)) +  D(\bar L) D(\bar K) - D(\bar K) D(\bar L) \right) \omega \\ 
%&=   \left( [\bar L, \bar K ]  
%+ (d D (\bar L, \bar K) + D([\bar L,\bar K])) +  D(\bar L) D(\bar K) - D(\bar K) D(\bar L) \right) \omega \\ 
%&= \left( D_{[\bar L, \bar K]} + (dD +  D\wedge D) (\bar L, \bar K) \right) \omega;
%\end{aligned}
 %\]
%i.e. $[D_{\bar L}, D_{\bar K}] = D_{[\bar L\bar K]}$ if and only if
%$dD (\bar L, \bar K) = - D\wedge D (\bar L, \bar K)$.
%
If $U\subset M$ is open, and $\omega^1, \dots, \omega^r$ of $\Gamma(U,E)$ form a frame, 
then $D$ is expressed over $U$ by its {\em connection matrix} $(D^\alpha_\beta)_{\alpha,\beta=1}^r$,  
\[ D_{L} \omega^\alpha = \sum_{\beta= 1}^r  D^\alpha_{\beta} (L) \omega^\beta.\]
The  forms $D^\alpha_{\beta}$ completely determine $D$, because for a
general section $\eta = \eta_\alpha \omega^\alpha$ (from now one we are using Einstein's summation convention) we have 
\[ \begin{aligned}
D_{L} \eta =  (L \eta_\beta  +  \eta_\alpha D^\alpha_\beta (L)) \omega^\beta.
\end{aligned} \]
After choosing a basis of vector fields $L_1, \dots, L_n \in \Gamma(U, \mathcal{V})$, we  shall also employ the notation
$D^\alpha_{j,\beta} = D^\alpha_\beta (L_j)$, and we shall write $[L_j, L_k] = C_{j,k}^\ell L_\ell$. It follows that the condition $[D_{ L}, D_{ K}] = D_{[ L,  K]}$ for all $ L,  K$ is equivalent to the
system of PDEs:
\begin{equation}\label{e:compatibilitycondition}
	 L_j D^\alpha_{k, \beta}  - 
	 L_k D^\alpha_{j, \beta}  
	+ D^\alpha_{k,\gamma} D^\gamma_{j,\beta}  - 
	D^\alpha_{j,\gamma} D^\gamma_{k,\beta} = C_{j,k}^l D^\alpha_{l,\beta} 
\end{equation}
for all $j,k = 1,\dots,n$, and $\alpha,\beta = 1,\dots, r$.
\begin{defn}
We say that a (distribution) section $\eta\in \Gamma(U,E)$ is a solution (over $U$) if $D_{ L} \eta =0$ (in the distributional sense) for every $ L \in \Gamma(U,\mathcal{V})$. Equivalently, $\eta = \eta_\alpha\omega^\alpha $ is a solution if its components
$\eta_\alpha$ satisfy the following system of first order PDEs:
\[ L_j \eta_\alpha = - D_{j,\alpha}^\beta  \eta_\beta, \quad j=1,\dots, n, \quad \alpha=1,\dots,r.\] 	
\end{defn}
Thus, the principal symbol of the bundle operator necessarily agrees with the vector fields in $\mathcal{V}$, but the solution condition also involves a zeroth order term. 
We can now reformulate Definition~\ref{def:infaut}: Any real vector field $X$ can be considered
as a section of $\holforms(M)^*$, and it is an infinitesimal automorphism if and only if $X$ is a solution of the canonical
structure on $\holforms(M)^*$. 
%
%%%%%%%%%%%%%%%%%%%%%%%%%%%%%%%%%%%
%%%%%%%%%%%%%%%%%%%%%%%%%%%%%%%%%%%
%
\subsection{The canonical involutive structure of a $\mathcal{V}$ bundle} % (fold)
\label{sub:the_canonical_cr_structure_of_a_cr_bundle}
%
%%%%%%%%%%%%%%%%%%%%%%%%%%%%%%%%%%%
%%%%%%%%%%%%%%%%%%%%%%%%%%%%%%%%%%%
%

% subsection the_canonical_cr_structure_of_a_cr_bundle (end)
It is an interesting question to ask whether or not a bundle $E$ as above
can be endowed with an involutive structure in such a way that sections solutions
are exactly the compatible maps from $M$ to $E$. 

Let $\eta = \eta_\alpha \omega^\alpha$ be a section over an open set $U$ over
which the bundle trivializes, with bundle coordinates $(x,\omega_1, \dots, \omega_r) \in U \times \C^r$. Let $L$ be a section of $\mathcal{V}$. Then we can write $\eta_\ast L$ as
\begin{equation*}
	\eta_\ast L = L + \sum_{\alpha=1}^r L\eta_\alpha \frac{\partial}{\partial \omega_\alpha} + \sum_{\alpha=1}^r L \bar{\eta}_\alpha\frac{\partial}{\partial \bar{\omega}_\alpha}.
\end{equation*}
If we assume $\eta$ is a solution, then 
\begin{align*}
	\eta_\ast L & = L  - \sum_{\alpha=1}^r D_\alpha^\beta(L)\eta_\beta \frac{\partial}{\partial \omega_\alpha} + \sum_{\alpha=1}^r L \bar{\eta}_\alpha\frac{\partial}{\partial \bar{\omega}_\alpha}\\
	& =  L  - \sum_{\alpha=1}^r D_\alpha^\beta(L)\omega_\beta(\eta) \frac{\partial}{\partial \omega_\alpha} + \sum_{\alpha=1}^r L \bar{\eta}_\alpha\frac{\partial}{\partial \bar{\omega}_\alpha}.
\end{align*}
We therefore define $\mathcal{V}_E \subset \mathbb{C} TE$ to consist of all tangent vectors of the form
\begin{equation}\label{eq:defn_V_E}
	 L - \sum_{\alpha} D^\beta_{\alpha} ( L) \omega_\beta \dop{\omega_\alpha} + X^{(0,1)}, 
\end{equation}
where $X^{(0,1)}$ denotes a $(0,1)$ vector field in $\mathbb{C}^r$ with $\mathcal{C}^\infty$-smooth coefficients in $U\times\mathbb{C}^r$, \textit{i.e.} $X^{(0,1)}$ is a vector field of the form
\[X^{(0,1)} = \sum_{\alpha=1}^r a_\alpha \frac{\partial}{\partial \bar{\omega}_\alpha},\]
and $a_1,\dots,a_r \in \mathcal{C}^\infty(U\times\mathbb{C}^r)$. We claim that $\mathcal{V}_E$ is well-defined, \textit{i.e.} it's definition is independent of the frame $\{\omega^1 ,\dots, \omega^r \}$. So let $\tilde \omega^\alpha = T^\alpha_\beta \omega^\beta$ be a different frame. Denoting $\tilde{T}$ the inverse of $T$, \textit{i.e.} $T_\alpha^\beta \tilde{T}_\beta^\gamma = \tilde{T}_\beta^\gamma  T_\alpha^\beta = \delta_{\alpha,\gamma}$, then the coordinate function relate as ${\tilde \omega}_\alpha = {\tilde T}_\alpha^\beta \omega_\beta $. The connection matrix transforms as follows: for every section $L$ of $\mathcal{V}$, and every $\alpha,\beta =1, \dots, r$, 
\[\tilde{D}_\beta^\alpha(L) = \big( LT_\gamma^\alpha + T_\delta^\alpha D_\gamma^{\delta}(L)\big) \tilde{T}_\beta^\gamma.\]
We also have that the vector fields $\partial/\partial x_j$ and $\partial/\partial \omega_\beta$ transform as
\[\dop{x_j} \rightsquigarrow \dop{x_j}  + \sum_\alpha\dopt{{\tilde T}_{\alpha}^\beta}{x} \omega_\beta \dop{{\tilde \omega}_\alpha} +  \sum_\alpha\dopt{{\bar{\tilde T}}_{\alpha}^\beta}{x} \bar{\omega}_\beta \dop{{\bar{\tilde{\omega}}}_\alpha}, \quad \dop{\omega_\beta} \rightsquigarrow \sum_\alpha {\tilde T}_{\alpha}^\beta \dop{{\tilde \omega}_\alpha}, \]
Then a vector field of the form \eqref{eq:defn_V_E} in the new coordinates $(x,\tilde{\omega})$ can be written as
\[ L - \sum_\gamma (D_\alpha^\beta(L)\tilde{T}_\gamma^\alpha  -  L \tilde{T}_\gamma^\beta)\omega_\beta \frac{\partial}{\partial \tilde{\omega}_\gamma} + \tilde{X}^{(0,1)}.\]
Now using that $(LT_\gamma^\beta)\tilde{T}_\alpha^\gamma = -T_\gamma^\beta L \tilde{T}_\alpha^\gamma$, one can check that 
\[ \tilde{D}_\beta^\gamma(L)\tilde{\omega}_\beta = (D_\alpha^\beta(L)\tilde{T}_\gamma^\alpha  -  L \tilde{T}_\gamma^\beta)\omega_\beta,\]
and therefore $\mathcal{V}_E$ is well defined. We claim that $\mathcal{V}_E$ defines an involutive structure on $E$. So let $L$ and $K$ be two sections of $\mathcal{V}$. Then by \eqref{e:compatibilitycondition} we have that
\[
\begin{aligned}
	&\left[  L - \sum_\alpha D_{ \alpha}^\beta ( L) \omega_\beta \dop{\omega_\alpha} + X^{(0,1) } , 
	K - \sum_\alpha D_{ \alpha}^\beta ( K) \omega_\beta \dop{\omega_\alpha} + Y^{(0,1) } \right] \\ &= [ L,  K] -\sum_\alpha \left(  L D_{ \alpha}^\beta (K) -  K D_{ \alpha}^\beta (L) + D_{\gamma}^\beta (L) D_{\alpha}^\gamma (K) - D_{\gamma}^\beta (K) D_{\alpha}^\gamma (L) \right)\omega_\beta \dop{\omega_\alpha} + Z^{(0,1)} \\
	&=
	[L, K] -\sum_\alpha  D_{ \alpha}^\beta ([L, K]) \omega_\beta \dop{\omega_\alpha} + Z^{(0,1)}, \\
\end{aligned}
\]
\textit{i.e.} $[\crb_E, \crb_E ] \subset \crb_E$.
We therefore have: 
\begin{prop}\label{pro:existenceVE}
If $E$ is a $\mathcal{V}$-bundle over $M$, then there 
exists a (unique) involutive structure $\mathcal{V}_E$ on $E$ with the property that the projection 
$E\to M$ is $(\mathcal{V}_E,\mathcal{V})$-compatible, and such that a section $\sigma \colon M \to E$ is 
a solution of $\mathcal{V}$ if and only if the map $\sigma \colon (M,\crb) \to (M,\crb_E)$
is $(\mathcal{V},\mathcal{V}_E)$-compatible. The rank of $\crb_E$ is $\text{rank}\,\mathcal{V} + r$. 
\end{prop}

We remark that if there are no compatible sections, then of course any structure on $E$ will satisfy
the conclusion of the proposition, however, if we have enough independent sections, then the structure on $E$ is unique. This  is why we put ``unique'' in parentheses in the formulation of the proposition.

\begin{prop}
For every $p \in M$, $\dim \left(\mathcal{V}\big|_p \cap \overline{\mathcal{V}}\big|_p\right) = \dim \left(\mathcal{V}_E\big|_p \cap \overline{\mathcal{V}_E}\big|_p\right)$.
\end{prop}
\begin{proof}
So let $p \in M$, and let $L$ be a section of $\mathcal{V}$ in a neighborhood of $p$, and consider 
\[ \mathcal{L} =  L - \sum_{\alpha} D^\beta_{\alpha} ( L) \omega_\beta \dop{\omega_\alpha} + X^{(0,1)}\]
the associated section of $\mathcal{V}_E$. Then $\bar{\mathcal{L}}_p \in \mathcal{V}_E\big|_p$ if and only if $\bar{L}_p \in \mathcal{V}\big|_p$ and 
\[ \overline{X^{(0,1)}} = - \sum_{\alpha} D^\beta_{\alpha} ( \bar{L}) \omega_\beta \dop{\omega_\alpha},\]
and this condition fixes the $(0,1)$ component of $\mathcal{L}$.
\end{proof}
\begin{cor}
If $\mathcal{V}$ is CR then $\mathcal{V}_E$ is also CR.
\end{cor}
\begin{prop}
If $\mathcal{V}$ is elliptic then $\mathcal{V}_E$ is also elliptic.
\end{prop}
\begin{proof}
Since every $(0,1)$ vector on the fiber belongs to $\mathcal{V}_E$, it is clear that $\mathcal{V}_E + \overline{\mathcal{V}_E} = \mathbb{C} T E$, given that $\mathcal{V} + \overline{\mathcal{V}} = \mathbb{C} T M$.
\end{proof}
\begin{cor}
If $\mathcal{V}$ is complex then $\mathcal{V}_E$ is also complex.
\end{cor}

% TODO: Discuss
% If on the other hand, $\tilde\crb_E$ is a CR structure 
% bundle on $E$ of such that  $\crdim E = \crdim M + e$, and 
% for which the projection to $M$ is CR, 
% then we can try to construct a derivative
%  operator $D$
% such that every section $\sigma \colon M\to E$ which 
% is CR as a map $(M,\crb)$ to $(M,\tilde \crb_E)$ is a CR section
% with respect to the operator $D$.

%  We first decompose any vector
% $X$ taking values in $\tilde \crb_E$ into 
% a horizontal and a vertical component,
% \[ X = H + V^{(1,0)} + V^{(0,1)} , \]
% where $H\in \C T M$, and $V\in \C T E $, and 
% write in coordinates as before
% \[ V^{(1,0)} = \sum_\alpha A_\alpha \dop{\omega_\alpha}.  \]

% For a given CR vector $\bar L \in \crb$, we can choose CR vectors $X^1, \dots, X^e$

%
%%%%%%%%%%%%%%%%%%%%%%%%%%%%%%%%%%%
%%%%%%%%%%%%%%%%%%%%%%%%%%%%%%%%%%%
%
\subsection{Constructions with $\mathcal{V}$-bundles: Dual bundles and tensor products} % (fold)
\label{sub:constructions_with_cr_bundles}
%
%%%%%%%%%%%%%%%%%%%%%%%%%%%%%%%%%%%
%%%%%%%%%%%%%%%%%%%%%%%%%%%%%%%%%%%
%
Just like the extension of a connection to the tensor algebra, we can 
use natural operations on $\mathcal{V}$-bundles. 
As a first example, given a $\mathcal{V}$-bundle $E \to M$, with derivative operator $D$, the dual 
bundle $E^*$ is also a $\mathcal{V}$-bundle if one equips it with the 
natural dual operation
\[ (D^*_{ L} \eta) (\omega) = 
 L (\eta (\omega)) - \eta (D_{L} \omega)  , \quad \omega \in \Gamma (U,E), 
\, \eta \in \Gamma (U,E^*), \, L \in \Gamma (U,\crb),
 \]
 where $U\subset M$ is any open set.
 In order to see that this 
 definition makes sense, we check  that the expression on the right hand side
 is actually linear in $\omega$: 
 \[ \begin{aligned}
 L  (\eta (a \omega)) - \eta (D_{ L}  (a \omega))  & =  {L} (a (\eta (\omega))) - 
 \eta ( ({L} a) \omega + a D_{L} \omega ) \\ 
 & =   ({L} a) \eta (\omega) + 
 a L (\eta (\omega)) - ({L} a) \eta ( \omega ) - a \eta( D_{L} \omega ) \\
 & =  a \big(L (\eta (\omega)) -  D_{L} (\eta (\omega)))\big),
 \end{aligned}
 \]
 and therefore the above definition defines $D^*_{L} \eta \in \Gamma (U, E^*)$. Let us also verify that $D^*$ satisfies the Leibniz rule: 
 \[ 
\begin{aligned}
(D^*_{L} a \eta) (\omega) &= 
 L (a \eta (\omega)) - a \eta (D_{L} \omega)  \\
&= ( a D^*_{L}  \eta + (L  a) \eta) (\omega).
\end{aligned}
 \]
 Last we check that $D^*$ satisfies
 that $[D^*_{L}, D^*_{K}] = D^*_{[L, K]} $:
 \begin{align*}
 \big([D_L^*, D_K^*]\eta \big)(\omega) &= D_L^*\big(D_K^*\eta\big)(\omega) - D_K^*\big(D_L^*\eta\big)(\omega) \\
 & = L\big(D_K^*\eta\big)(\omega) - D_K^*\eta(D_L\omega) -  K\big(D_L^*\eta\big)(\omega) +D_L^*\eta(D_K\omega)\\
 & = L\big(K\eta(\omega) - \eta(D_K\omega)\big) - K\eta(D_L\omega) + \eta(D_KD_L \omega)\\
 &\quad - K\big(L\eta(\omega) - \eta(D_L\omega)\big) + L\eta(D_K\omega) - \eta(D_LD_K \omega)\\
 & = [L,K]\eta(\omega) - \eta(D_{[L,K]}\omega).
 \end{align*}

Now let $E$ and $F$ be $\mathcal{V}$-bundles over $M$ with derivative operators $D^E$ and $D^F$, 
respectively. We then define a derivative operator on $E \otimes F$  by
\[ D_L^{E\otimes F} \omega\otimes \eta = D_L^E \omega \otimes \eta + \omega \otimes D_L^F \eta.   \]
One checks that the operator defined in this way satisfies all of the required properties.
In particular, if $E$ and $F$ are as above, then $\Hom (E,F)= E^* \otimes F $ is a $\mathcal{V}$-vector 
bundle, with the bundle operation $D$ defined by 
\[ (D_{L} A) (\omega) = D^F_{L} (A(\omega)) - A (D^E_{L} \omega).   \]
% subsection constructions_with_cr_bundles (end)

%
%%%%%%%%%%%%%%%%%%%%%%%%%%%%%%%%%%%
%%%%%%%%%%%%%%%%%%%%%%%%%%%%%%%%%%%
%
\subsection{Subbundles} % (fold)
\label{sub:subbundles}
%
%%%%%%%%%%%%%%%%%%%%%%%%%%%%%%%%%%%
%%%%%%%%%%%%%%%%%%%%%%%%%%%%%%%%%%%
%
We say that  a subbundle $F \subset E$ of a $\mathcal{V}$-vector bundle 
$E$ over $M$ is a $\mathcal{V}$-subbundle if the differential operator 
$D^E_{L}$ has the property that it maps sections of 
$F$ to sections of $F$, i.e. if 
\[ D^E_{ L} \omega \in \Gamma (U,F), \quad \omega \in \Gamma (U,F), \]
for every open set $U\subset M$.
In that case, the restriction of $D^E$ to sections of $F$ of 
course defines a $\mathcal{V}$-vector bundle on $F$. 

\begin{lem}\label{lem:perp} Let $E$ be a $\mathcal{V}$-bundle, and
	let $F\subset E$ be a vector subbundle. Then $F^\perp \subset E^*$ 
	is a $\mathcal{V}$-subbundle of $E^*$ if and only if $F$ is a $\mathcal{V}$-subbundle
	of $E$. 
\end{lem}

\begin{proof}
	If $\omega\in \Gamma(U,F)$
	is a section of $F$, 
	and $\omega^*\in \Gamma(U,F^\perp)$ is a section of $F^\perp$,
	we have 
	\[ \omega^* (D_{L} \omega ) =  - (D_{L}^* \omega^* ) (\omega) \]
	and hence  the claim follows.
\end{proof}

% subsection subbundles (end)
%
%%%%%%%%%%%%%%%%%%%%%%%%%%%%%%%%%%%
%%%%%%%%%%%%%%%%%%%%%%%%%%%%%%%%%%%
%
\subsection{Factor bundles}
%
%%%%%%%%%%%%%%%%%%%%%%%%%%%%%%%%%%%
%%%%%%%%%%%%%%%%%%%%%%%%%%%%%%%%%%%
%
The factor bundle $E/F$ is a $\mathcal{V}$-bundle with the derivative operator
induced from $E$ if $F$ is a $\mathcal{V}$-subbundle. For the dual bundles we have
the equalities 
\[ \left( E/F \right)^* = F^\perp, 
\quad F^*  = \left( E^* / F^\perp \right) \]
not only as vector bundles, but also as $\mathcal{V}$-bundles. 

%
%%%%%%%%%%%%%%%%%%%%%%%%%%%%%%%%%%%
%%%%%%%%%%%%%%%%%%%%%%%%%%%%%%%%%%%
%
\subsection{Extension of the derivative operator to $\crb + \bar \crb$} 
%
%%%%%%%%%%%%%%%%%%%%%%%%%%%%%%%%%%%
%%%%%%%%%%%%%%%%%%%%%%%%%%%%%%%%%%%
%
If $E$ is a $\mathcal{V}$-bundle over $M$, it is tempting to try to define
\[ D_{\bar L} \omega = \overline{ D_{ L} \bar \omega }. \]
This needs a conjugation operator on $E$, or equivalently, a 
maximally totally real subbundle $R\subset E$. We shall thus 
assume that we are given a 
real subbundle $R$ of $E$ with the property that
\[ E = R \oplus i R; \]
the conjugation operator on $E$ is then defined by $v + i w \mapsto v - i w$. 
Indeed, the following holds. 
\begin{lem}
Let $R$ be a maximally real subbundle of $E$. Then $\sigma_R (v+i w) = v-iw$ for 
$v,w \in R$ defines an antilinear involution of $E$. On the other hand, 
given an antilinear involution $\sigma$ of $E$, the set $R_\sigma = \left\{ v \colon \sigma (v) = v \right\}$ is a maximally real subbundle. This sets up a one-to-one
correspondence of antilinear involutions of $E$ on the one hand and maximally real subbundles of $E$ on the other hand.
\end{lem}

% Many of the bundles we have considered so far are 
% complexifications of real bundles and thus have natural choices for 
% such a subbundle $R$: The real (co)tangent bundle $TM \subset \C TM$
% ($T^*M \subset \C T^* M$). 
If $E$ is endowed with a maximally totally 
real subbundle $R\subset E$, so is $E^*$; in fact, 
we can identify 
\[  R^*_x = \{ w \in E^*_x : w(v) \in \R \text { for all } v\in R_x  \}.
\]
In coordinates, if we assume that $R$ is spanned by the 
$\R$-linear span of the sections $\eta^1 , \dots, \eta^r$, then in 
coordinates $\eta_1, \dots , \eta_r$ defined through
$\eta = \eta_\alpha \eta^\alpha$ we have $\sigma (\eta)_\alpha = \overline{\eta_\alpha}$. If we are using a different frame $\omega^\alpha = T^\alpha_\beta {\eta}^\beta $, $\eta^\alpha = {\tilde T}^\alpha_\beta \omega^\beta$ then the conjugation $\sigma$ has the following
coordinate expression:
\[ \sigma(\omega_\alpha \omega^\alpha) = (\overline{ \omega}_\alpha {\overline T}_\beta^\alpha \tilde T^\beta_\gamma) \omega^\gamma.  \]
In other words, $\omega$ transforms  by $\omega\mapsto M \bar \omega$ with the matrix $M = T^{-1} \bar T$, if we identify sections of $E$ with column vectors
\[ \omega = \begin{pmatrix}
	\omega_1 \\ \vdots \\ \omega_r
\end{pmatrix}; \]
$T_\beta^\alpha$ is then identified with the matrix $T$ by taking $\beta$ as the 
row and $\alpha$ as the column index. 
 Note that $\bar M = M^{-1}$ and that every antilinear involution 
can be written in this way.

Given a maximally real subbundle (or equivalently, an involution $\sigma$), 
we define the extension of $D_{L}$ to $\bar \crb$ by the formula
\[ D_{\bar L} \omega = \sigma (D_{L} \sigma \omega). \]
In coordinates, this extends $D_\alpha^\beta$ to $\bar \crb$ by 
\[ D_{\bar L} \omega = (\bar L + D(\bar L)) \omega, \qquad D(\bar L) = M \bar L \bar M + M \overline{D(L)} \bar M. \]
%
%%%%%%%%%%%%%%%%%%%%%%%%%%%%%%%%%%%
%%%%%%%%%%%%%%%%%%%%%%%%%%%%%%%%%%%
%
\section{Integrability}
%
%%%%%%%%%%%%%%%%%%%%%%%%%%%%%%%%%%%
%%%%%%%%%%%%%%%%%%%%%%%%%%%%%%%%%%%
%

This section is not strictly speaking necessary for the further development, but illustrates nicely the use of
the bundle calculus introduced above. 

\begin{defn}
	\label{def:integrability} Let $E$ be a $\mathcal{V}$-bundle over $M$. 
	We say that $E$ is locally integrable at $p$ if there 
	exist sections solutions defined on some open neighborhood $U$ of $p$, $\omega^1, \dots , \omega^e\in \Gamma (M,E)$, 
	such that $\{ \omega^1 (p), \dots \omega^r (p) \}$ spans
	$E_p$.  We say that $E$ is locally integrable if it is locally integrable 
	at every $p\in M$.
      \end{defn}

      \begin{rmk}
        Definition~\ref{def:integrability} is the reason we require that $D$ is partially flat, because if $D$ is locally
        integrable, it is partially flat. 
      \end{rmk}
We compute that for each $j = 1, \dots , n$, if 
$\sigma^1 , \dots , \sigma^r$ is a set of sections such that 
$\{ \sigma^1 (p) , \dots , \sigma^r (p) \}$ spans $E_p$, then we 
have for $\omega^\alpha = \sum_\beta \lambda^\alpha_\beta \sigma^\beta $ and for $j = 1,\dots ,n $ that
\[
\begin{aligned}
	D_{L_j} \omega^\alpha &= 
	D_{L_j} \sum_\beta \lambda^\alpha_\beta \sigma^\beta  \\
	&= \sum_\beta (L_j \lambda^\alpha_\beta) \sigma^\beta + 
	\sum_\beta \lambda^\alpha_\beta \left( D_{L_j}\sigma^\beta \right) \\
	&= \sum_\gamma ( L_j \lambda^\alpha_\gamma) \sigma^\gamma + 
	\sum_{\beta ,\gamma}\lambda^\alpha_\beta \left( D_{j,\gamma}^
	\beta \sigma^\gamma \right) \\ 
	&= \sum_\gamma \left(( L_j \lambda^\alpha_\gamma) + 
	\sum_{\beta}\lambda^\alpha_\beta D_{j,\gamma}^
	\beta \right) \sigma^\gamma.
\end{aligned}
\]
With the $e\times e$ matrix $\Lambda  = (\lambda_\gamma^\alpha)_{\alpha,\gamma}$
and $D_j = (D_{j,\gamma}^\beta)_{\gamma,\beta} $ we see that the equations
\[ 	D_{ L_j} \omega^\alpha  = 0, \quad j=1,\dots, n, \quad \alpha=1,\dots, e, \]
are equivalent to 
\[ L_j \Lambda = - \Lambda D_j, \quad j=1,\dots ,n. \]

% We shall now assume that if we consider the space of 
% CR sections of $E$, which we denote by $\Gamma_{CR} (M,E)$, has
% the property that it spans a subbundle $\tilde E \subset E$,
% i.e. that 
% \[ \dim_\C \Gamma_{CR} (M,E) (p) = \dim_C \{ \omega(p) \colon
%  \omega\in\Gamma_{CR} (M,E) \}\]
%  is constant for $p\in M$. The subbundle $\tilde E \subset E$
% is obviously a CR subbundle. 

We shall now examine certain inheritance properties of integrability.

\begin{prop}\label{pro:integrable2} If $E$ is locally integrable, 
and $F\subset E$ is a $\mathcal{V}$-subbundle, then $F$ is locally integrable. 
\end{prop}

\begin{proof}
Let $\omega^1, \dots, \omega^r$ be a local basis of sections solutions of $E$. 
Let $\eta^\beta = T_\alpha^\beta \omega^\alpha$ be a local basis
of sections of $F$, where $\beta = 1,\dots, f = \dim_\C F$; without 
loss of generality we can assume that $T_\alpha^\beta = \delta_\alpha^\beta$
for $\alpha,\beta = 1,\dots, f$.
%
%We can always change $\omega^1, \dots, \omega^r$ by a linear combination such that $T_\alpha^\beta = \delta_{\alpha,\beta}$ if $\alpha,\beta = 1, \dots, f$, and $T_\alpha^\beta$ otherwise in a given point. Therefore in a neighborhood of this point we have that the matriz $(T_\alpha^\beta)_{\alpha,\beta = 1}^f$ is invertible, thus replacing $\eta^1, \dots, \eta^f$ by a smooth linear combination of them, we achieve the desired for $T_ \alpha^\beta$.
%
 Now $D_{ L} \eta^\beta$ is a linear combination of $\omega^{f+1} , \dots , \omega^r$ and also an element of $F$,
i.e. $D_{ L} \eta^\beta = 0$ for $\beta = 1,\dots, f$. 
\end{proof}

\begin{prop}
\label{pro:integrable2} If $F\subset E$ is  a $\mathcal{V}$-subbundle, then $F^\perp$
is locally integrable if and only if $(E/F)^*$ is. 
\end{prop}

\begin{prop}
\label{pro:integrable} Assume that  $F \oplus G = E$ for $\mathcal{V}$-subbundles $F, G$. If $G$ is locally integrable, then  $F^\perp$ is. 
\end{prop}
\begin{proof}
Let $p\in M$, and denote $f = \rank F$, $g= r-f$. We choose a basis $\omega^1 , 
\dots , \omega^f$, and $\omega^{f+1}, \dots \omega^r$ of sections of $F$ and $G$ near $p$, with $\omega^{f+\beta}$ being solutions, $\beta = 1,\dots,g$.  
 Now consider the dual forms $\omega_{f+1}, \dots, \omega_{r} \in E^* $, which form a basis of $F^\perp$.  We have 
 \[ D_{L}^* \omega_{f+\alpha} (\omega^{f+\beta}) = -\omega_{f+\alpha} (D_{L} \omega^{f+\beta}) = 0, \qquad \alpha,\beta = 1,\dots ,g \]
 and since $D_L \omega^\beta$ is a section of $F$, for $F$ is a $\mathcal{V}$-subbundle, we also have
 \[ D_{L}^* \omega_{f+\alpha} (\omega^\beta) = \omega_{f+\alpha} (D_{L} \omega^\beta) = 0, \qquad \alpha = 1,\dots,g ,\quad \beta = 1,\dots, f. \]
 % \omega_{f+\alpha} (D({L })^\beta_\gamma \omega^\gamma)
 Hence, $D^*_{L} \omega_{f+\alpha} = 0$ for $\alpha = 1, \dots , g$, and 
 so $F^\perp$ is locally integrable. 
\end{proof}

\section{$\crb$-modules} % (fold)
\label{sec:cr_modules}
Assume that $E$ is a $\crb$-bundle over $M$, and consider the sheaf $\Gamma(\cdot, E)= \mathcal{E}$, or more generally
the sheaf $\mathcal{D}'(\cdot,E)=:\mathcal{D}'_E (\cdot)$, as a sheaf of modules over $\cinfty(M)$.
We say that a submodule 
$\mathcal{S} \subset \mathcal{E}$ is a $\crb$-module if for every $U\subset M$ and any section $\omega\in \mathcal{S}(U)$ and vector
field $L\in \Gamma(U,\crb)$ we have that $D_L \omega \in \mathcal{S} (U)$. Here are some examples.

\begin{exa}
  Assume that $F\subset E$ is a $\crb$-subbundle. Then $\Gamma(\cdot,F) \subset \mathcal{E}$ is a $\crb$-bundle. 
\end{exa}

\begin{exa}
  For any $U\subset M$, let $\ker D (U) := \{ \omega \in \mathcal{E} (U) \colon D_L \omega = 0 \, \forall L \in \Gamma (U,\crb) \}$.
  Then $\ker D$ is {\em not} a $\crb$-module in general, as typically $D_L a \omega = La \omega\neq 0$ for $\omega \in \ker D (U)$. 
\end{exa}

We say that a $\crb$-module $\mathcal{S}$ is locally integrable if for every open $U\subset M$ there exist finitely many
$\crb$-sections $\sigma_1, \dots, \sigma_r \in \mathcal{S}(U)$ spanning $\mathcal{S} (U)$ (over $\cinfty(M)$).

Let $\mathcal{S}\subset \mathcal{E}$ be an arbitrary submodule. Then there exists a smallest $\crb$-submodule $\hat{\mathcal{S}}$ which contains $\mathcal{S}$
(since the intersection of two $\crb$-submodules is again one); we shall call this the $\crb$-hull of $\ssheaf$.
We can describe $\hat{\mathcal{S}}$ in the following way.

\begin{lem}
  \label{lem:sclosure} Given a submodule  $\mathcal{S}\subset \mathcal{E}$, we define an ascending chain
  of submodules inductively by setting $\mathcal{S}_0 = \mathcal{S}$ and for every $j>0$ and $U\subset M$
  \[ \mathcal{S}_j =
    \left\{ \sum_\ell a_\ell \omega_\ell + \sum_k b_k  D_{\bar L_k} \eta_k \colon a_\ell, b_k \in \cinfty (U), \, \omega_\ell, \eta_k \in \mathcal{S}_{j-1} (U), \, L_k \in \Gamma(U,\crb) \right\}. \]\
  Then $\hat{\mathcal{S}} = \bigcup_{j=0}^\infty \mathcal{S}_j$. Since $ \dim \ssheaf(p)$ is upper semi-continuous on $M$,
  there exists a dense open subset $\hat M\subset M$ which decomposes as $\hat M = \cup \hat M_j$ with $\dim \ssheaf(p) = j$,
  where $1\leq j \leq \dim_\C E=:r$. It then holds that on a dense open subset $\tilde M_j \subset \hat M_j$, $\hat \ssheaf|_{\tilde M_j} =
  \ssheaf_{j}|_{\tilde M_j}$. In particular, $\hat \ssheaf$ is generically finitely generated if $\mathcal{S}$ is. 
\end{lem}
\begin{proof}
The first part of the lemma is a standard verification, as the union of the $\mathcal{S}_j$ obviously is a $\crb$-submodule, and every
$\crb$-submodule containing $\mathcal{S}$ necessarily contains all of the $\mathcal{S}_j$.

Assume that on an open subset $U$, we have $\mathcal{S}_{j+1} (p) \subset \mathcal{S}_j (p)$ for every $p\in U$. Then in
particular, $\mathcal{S}_{j+1}(U) \subset \mathcal{S}_j(U)$, so that we have $\hat \ssheaf (U) = \ssheaf_j (U)$ by induction. It follows
that generically the dimensions of the spaces spanned by $\mathcal{S}_j (p)$ need to be strictly increasing. 
\end{proof}

\begin{exa}
  The genericity assumption is necessary, even in the real-analytic CR setting. Indeed, let $\Phi(z)$ be a holomorphic entire
  function on $\C$ with zeroes of order $k$ exactly at $k \in \N$, and consider the hypersurface $M \subset \C^2$ defined by  $\imag w= \real z \overline{\Phi(z)}$.
  Then the CR vector field on $M$ is given by $L = \dop{\bar z} + i (\Phi(z) + z \overline{\Phi'(z)}) \dop {\bar w}$, and
  a characteristic form is given by $\theta = 2i \partial (\imag w - \real z \overline{\Phi (z)}) = dw - (\bar z \Phi'(z) + \overline{\Phi(z)}) dz $. It follows that
  \[ (\mathcal{L}_L)^k \theta = - \overline{\Phi^{(k)} (z)} dz, \quad k\geq 2,  \]
  and so for $k\geq 2$, we need $k$ derivatives to obtain $dz$. In this example, if we let $\ssheaf$ consist of the sections of the
  characteristic bundle $\C \charforms{M}$, and we use the canonical structure of $\holforms(M)$, we have $\ssheaf_1 (U) = \holforms M$ where
  $U = \{ (z,w) \in M \colon z \notin \N \}$ and $\ssheaf_j (k) = T'M$ only if $j\geq k$. 
\end{exa}

We now turn to modules of smoothness multipliers. For this, we assume that we are given a $\crb$-submodule
$\Omega$  of the sheaf of generalized sections $\mathcal{D}'_E $.

\begin{defn}
  The module of smoothness multipliers of $ \Omega$ is defined by
  \[ \mathcal{S}(\Omega) (U) = 
    \left\{
      \eta \in \mathcal{E}^* (U) = \Gamma (U, E^*) \colon  \eta (\omega) \in \cinfty(U) \, \forall
      \omega \in \Omega
    \right\} \]
\end{defn}

We first note that this sheaf is well-defined,
since any section  $\omega \in \mathcal{D}'_E(U)$ can be evaluated along a
smooth section of $\eta \in  \mathcal{E}^*(U)$; in general, this yields a distribution $\eta(\omega)\in\mathcal{D}'(U)$, so it
makes sense to say that the section is smooth.

\begin{lem}\label{lem:smoothnessV}
  For any $\crb$-module $\Omega$ it holds that  its module of
  smoothness multipliers $\mathcal{S}(\Omega)$ is also a $\crb$-module. 
\end{lem}

\begin{proof}
  We have $D_L^* \eta (\omega) = L \eta(\omega) - \eta(D_L \omega) \in \cinfty(U)$
  if $L \in \Gamma (U,\crb)$, $\eta \in \mathcal{S}(M)$, and $\omega \in \Omega $, hence $D_L^* \eta \in\mathcal{S}(M) $ if $\eta\in \mathcal{S}(M)$. 
\end{proof}

Observe that for every  $\crb$-module $\Omega$ we have the
following dichotomy: If $p\notin \singsupp \Omega$, then there exists a neigbourhood $U$ of
$p$ such that $\mathcal{S}(\Omega)|_U = \Gamma (U,E^*) $; if $p\in \singsupp \Omega$, then
there exists a neighbourhood $U$ of $M$ such that $\mathcal{S}(\Omega)|_U \subsetneq \Gamma(U,E^*)$.
In the second case, a dense open subset
$V\subset U$ can be decomposed as $V= \bigcup_{j=0}^{r-1}  V_j$ such
that over each $V_j$, $\mathcal{S}(\Omega)|_{V_j} = \Gamma(\cdot,E_j^*)$, where
$E_j^* \subset E^*|_{V_j}$ is a $\crb$-subbundle of rank $j$ defined over $V_j$.

This translates to the following statement: If $\Omega$ is not smooth in a neighbourhood of $p$,
then every $q$ in a dense open subset of a neighbourhood of $p$ has a neighbourhood
$U$ on which there exists a $\crb$-subbundle $\tilde E \subset E|_U$ such that
$ \Gamma (U, \tilde E) \subset \Omega|_U $.
Indeed, assume both $\Omega$ and $\mathcal{S}(\Omega)$ are of constant rank over $U$.
Then note that
$\Omega^\perp \subset \mathcal{S}(\Omega) $ by definition,
so if we let $\tilde E= (E_j^*)^\perp$,
then $(\Omega|_U)^\perp \subset \Gamma(U, E_j^*)$,
so $\Gamma(U, \tilde E^\perp) \subset \Omega|_U$.
Furthermore, in that case the quotient section module
$\pi(\Omega|_U) \in \Gamma(U,E/\tilde E)$ is smooth.

One might be tempted to generalize this last cumbersome translation to a more
global statement, but in our setting this is not possible because the ``smoothness ranks''
of $\Omega$ can change. However, one obtains a simple criterion which forces smoothness
of $\Omega$.

\begin{defn}
  We say that a submodule $\mathcal{R} \subset \Gamma(\cdot, E)$ is {\em nondegenerate}
  if $ \hat{ \mathcal{R}} =  \Gamma(\cdot, E)$. We say that it is $k$-{\em nondegenerate}
  if $\mathcal{R}_{k-1} \subsetneq \mathcal{R}_k = \Gamma(\cdot, E)$. 
\end{defn}

\begin{prop}\label{prop:nondegenerate}
  Assume that $\Omega$ is a $\crb$-bundle and that
  $ \mathcal{R} \subset \mathcal{S}(\Omega)$ for some nondegenerate
  $\mathcal{R} \subset \Gamma(\cdot,E^*)$. Then $\Omega$ is smooth. 
\end{prop}

\begin{proof}
  Since $\Omega$ is a $\crb$-bundle, so is $\mathcal{S} (\Omega)$ by Lemma \ref{lem:smoothnessV}. Hence   $ \Gamma(\cdot, E^*) = \hat{\mathcal{R}} \subset \mathcal{S}(\Omega)$. 
\end{proof}

% Let's assume that $h\colon M\to M'$ is $(\crb,\crb')$-compatible. Then we define the
% infinitesimal deformations of $h$ as follows. Since $h$ is
% compatible, we have $h^* \eta \in \Gamma(h^{-1}(U),\holforms(M) ) $ if $\eta \in \Gamma(U,\holforms{M'})$.
% If $h = h_0$ is part of a family of compatible maps $h_t$ depending at least $\diffable{1}$ on $t\in (-\varepsilon, \varepsilon)$,
% we can differentiate the equation
% \[ h^*_t \eta (L) = 0  \]
% with respect to $t$. Writing in coordinates $\eta = \sum \eta^j dy_j$, $h_t = (h_{t,1},\dots, h_{t,n})$ and
% $L= \sum L_j \dop{x_j}$ we get (with the shorthand notation $\dot h = (\dot h_1, \dots, \dot h_n)= \frac{d}{dt}|_{t=0} h_t  \in \Gamma(\cdot,h^* TM')$)
% \[
%   \begin{aligned}
%     \frac{d}{dt}\biggr|_{t=0} h^*_t \eta (L) &= \frac{d}{dt}\biggr|_{t=0} \sum_{k} (\eta^k\circ h_t) L h_{t,k} \\
%                                              & = d\eta^k (\dot h ) Lh_k + (\eta^k \circ h) L\dot h_k \\
%     &= 
%   \end{aligned}
% \] 

% Section cr_modules (end)

\section{The characteristic sheaf}

We have already introduced the natural derivative operator on the bundle $\holforms(M)$, which we
are going to regard as an operator on the sheaf $\mathcal{T}'$ defined by
$\holsheaf(U) = \Gamma (U, \holforms(M))$. The characteristic sheaf $\mathcal{T}^0 \subset \mathcal{T}'$ of $M$ is defined by
\[ \charsheaf(U) := \left\{ \omega \in \holsheaf(U) \colon \omega_p \in\charforms_p M \right\}.
\]
We can now consider the $\crb$-hull $\hat{\mathcal{T}}^0$ and apply the decomposition described in \cref{lem:sclosure} to it,
yielding a disjoint union $ \hat M = \biguplus_{j=0}^{m} \hat M_j $ of the dense, open subset $\hat M$ into open
subsets on which the rank of $\hat{\mathcal{T}}^0$ is constant, meaning
\[ \hat \charsheaf(U) = \Gamma(U, E_j) , \quad  U \subset \hat M_j,   \]
where $E_j \subset \holforms(M)|_{\hat M_j} $ is a $\crb$-subbundle of $\holforms(M)$.

\begin{defn}\label{def:degeneracy}
  We say that $M$ is constant degeneracy $d$  at $p$ if $p\in \hat M_{m-d}$. We say that $M$ is nondegenerate at $p$ if it is (constantly) $0$-degenerate at $p$, i.e.
  if $p\in \hat M_{m}$. We say that $M$ is nondegenerate if $M= \hat M_m$. 
\end{defn}
An example of a nondegenerate structure was given in \cref{sub:example}.

We will need to introduce some more  terminology related to infinitesimal automorphisms. It turns out that
not all components of an infintesimal automorphisms will be smooth, and here we identify two obstructions:
First, the nondegeneracy conditions introduced in \cref{def:degeneracy} and second, ``purely real'' components of
infinitesimal automorphisms, which we are now going to introduce. In order to facilitate
notation, we define for any real vector field $X \in \vectorfields(U) := \Gamma(\cdot, TM)$)
an associated section $X^\sharp \in (\holsheaf)^*(U)$ by
\[ X^\sharp (\omega) = \omega(X).  \]
Note that any such $X^\sharp$ has $ X^\sharp(\imag \theta) = 0$ for $\theta \in \charsheaf(U)$ and that $X^\sharp$ is also defined for distributional vector fields. 

The kernel of the homomorphism $X \mapsto X^\sharp$ (where we consider $(\holsheaf)^*$ in a natural way as a real module) is
given by
\[ \mathfrak{R} (U) := \left\{ X \in \vectorfields (U) \colon X_p \in \crb_p \cap \bar \crb_p \quad \forall p\in U\right\}.
\]
Indeed, any such $X$ clearly satisfies $X^\sharp =0$, and on the other hand, if $X^\sharp =0$, then $\omega (X) = \bar \omega (X) = 0$
for any $\omega\in \holsheaf$, so $X\in \crb\cap \bar \crb$. Also note that $X^\sharp$ satisfies $\imag X^\sharp (\theta) = 0$ for
every $\theta\in \charsheaf$ and that any section $X$ of $\mathfrak{R}$ is automatically an infinitesimal automorphism:
we have $\omega([X,L]) = 0$ since $X\in\crb$ and $\crb$ is involutive. This corresponds to the fact that we do not, in
general, expect to be able to control fiber movements if we have positive-dimensional fibers on an open subset. Here are some examples.

\begin{exa}
  Consider the real involutive structure $\R_x $ generated by $\dop{x}$. Any map $x\mapsto \varphi (x)$ is an automorphism
  of this structure, and any real vector field $a(x) \dop{x}$ is an infinitesimal
  automorphism.

  A bit less trivial, we can consider product structures, for example
  the manifold $\R^4_{x,y,s,\xi}$ and the structure generated by $L=\dop{\bar z} -  i z \dop{s}$ (i.e. $z= x+iy$ and $w = s + i |z|^2$ are first
  integrals).  Any map $(x,y,s,\xi) \mapsto (x,y,s,\varphi (\xi)) $  is an automorphism
  of this structure, and any real vector field $a(\xi) \dop{\xi}$ is an infinitesimal
  automorphism. 
\end{exa}

Remember that $\imag X^\sharp (\theta)=0$ for $X \in \vectorfields$.
If we let $\mathfrak{S}\subset \holsheaf^*$ be the (real) submodule defined by 
 $Y\in \mathfrak{S} (U)$ if  $\imag Y (\theta) = 0$ for $\theta \in \charsheaf (U)$,
then we can put
\[ Y(\omega + \bar \eta) = Y(\omega) + \overline{Y(\eta)}, \quad \omega, \eta \in \holsheaf (U), \]
to extend $Y$ to $\holsheaf + \overline{\holsheaf}$ such that the extension
satisfies $\tilde Y(\bar \omega) = \overline{\tilde Y (\omega)}$. Indeed, if we are given two
representations $\omega_1 + \bar \eta_1 = \omega_2 + \bar \eta_2$, then $\omega_1 - \omega_2 = \bar \eta_2 - \bar \eta_1 \in \holsheaf(U) \cap \overline{\holsheaf}(U) $.
We claim that this means that the real and imaginary parts of $\omega_1 - \omega_2$ belong to $\charsheaf$. Assuming this claim, we get $Y(\omega_1) - Y(\omega_2)  =
\overline{Y(\eta_2)} - \overline{Y(\eta_1)} $, and so the action of $Y$ on $\holsheaf(U) + \overline{\holsheaf}(U)$ is well defined. In order to prove the claim,
let $\alpha\in \holsheaf(U) \cap \overline{\holsheaf}(U)$, then $\real \alpha (L) = \frac{1}{2} \left( \alpha (L) +\overline{\alpha (L)}  \right) =0$ as claimed.
By definition, this extension $\tilde Y$ of $Y$ satisfies $\tilde Y (\bar \omega) = \overline{\tilde Y(\omega)}$ as claimed.

 Let us now explain why we do not expect infinitesimal automorphisms to be smooth
at (constantly) degenerate points. This is due to the following observation. If $Y=X^\sharp $ is a section of $ \left( \faktor{\holsheaf}{\hat \charsheaf} \right)^* \simeq \ann \hat \charsheaf \subset \holsheaf^*$,
then in particular $Y(\theta) = 0$ for any $\theta \in \charsheaf$.  Thus generically, if $Y$ is such an infinitesimal automorphism, then $iY$ will be too; we assume that
$iY = (JX)^\sharp$ for some $JX$ (which, in the case of a CR structure, will actually be the complex structure operator applied to $X$). We now claim that for any solution $f$
of our structure, $fY$ is another solution; thus even smooth infinitesimal automorphisms might be ``as bad'' as solutions.
Indeed, if we write $f = \alpha + i \beta$, then
\[ \omega([\alpha X  -  \beta J X, L]) = \alpha \omega( [X,L]) + \beta \omega( [ JX,L]) - L (\alpha + i \beta) \omega(X) = 0. \]

\begin{exa}
  Consider the structure introduced in \cref{sub:example}. Remember that if $t = 0$, then $\mathcal{V}_{(x,y,0)} = \{a \frac{\partial}{\partial t} : a \in \mathbb{C} \}$,
  so $\mathcal{V}_{(x,y,0)} \cap \overline{\mathcal{V}}_{(x,y,0)} =  \{a \frac{\partial}{\partial t} : a \in \R \}$, 
  while $\mathcal{V}_{(x,y,t)}$. It follows that the sheaf $\mathfrak{R} = \{0\}$ (since every of its section is going to vanish outside of $t=0$),
  but there are nontrivial distributional infinitesimal automorphisms, such as $X=\delta_{0} \dop{t}$ (which is an element of $\distributions (M,\crb\cap\bar \crb)$ in
  this setting). Indeed, recall
  that we need to check that $L\omega (X) = d\omega (L,X)$ for any $\omega \in \holsheaf$; since $\holsheaf \R^3$ is spanned by
  $dZ = dx + i t^\ell dt$ and $dW = dy + i t^kdt$, and $L=\dop{t} - i t^\ell \dop{x} - it^k \dop{y}$, 
  this just boils down to checking that $L t^k \delta_0 = L t^\ell \delta_0 =0$, both of which are obvious.

  We will also use this example to shed some light on the image of the $\sharp$-operator. Note
  that the forms  $dZ , dW,d\bar Z, d\bar W$ spans $\holsheaf + \overline{\holsheaf}$, which is thus spanned by $dx, dy, t^k dt$ (if $k\leq \ell$).
  An element $Y$ of $\holsheaf^*$ given by $Y(dZ) = \alpha + i \beta$ and $Y(dW) = \gamma+ i\delta$
  satisfies $\imag Y(\theta) = 0$ if $t^k \beta = t^\ell \delta$; on the other hand, it is induced by
  by a vector field $X = a \dop{x} + b \dop{y} + c\dop{t}$ if $\alpha = a$, $\gamma = b$, $\beta = t^\ell c$, and $\delta = t^k c$.
  So even though $\imag Y(\theta) = 0$ is certainly necessary to be in the image of the $\sharp$-operator, it
  is definitely not sufficient. 
\end{exa}

This leads us to the following definition.

\begin{defn}
  We say that $M$ is regular at $p$ if for every neighbourhood $U$ of p and locally bounded vector field $X$ on $U$ giving rise to a smooth section $X^\sharp \in (\holsheaf)^*(U)$
  there exists an open neighbourhood $V$ of p such that $X|_V \in \vectorfields(V)$. 
\end{defn}

\begin{rmk}
  Regularity implies that $\dim \crb_q \cap \overline{\crb}_q =0$ almost everywhere near $p$. In particular,
   a regular nondegenerate structure is CR almost everywhere. We will discuss how to detect regularity later. 
\end{rmk}

We can now formulate the main theorem of this section and its corollaries.

\begin{thm}\label{thm:infinitesimals}
  Assume that $X$ is an infinitesimal automorphism of the involutive structure $(M, \crb)$ and that
  $X$ extends microlocally. Then there
  exists an open dense subset $\hat M$ of $M$ and a decomposition $\hat M= \biguplus_{j=0}^m \hat M_j $ such
  that on $\hat M_j$, we have that $X^\sharp \in \holsheaf$ is smooth modulo the sections $\ann \hat{\charsheaf}$ of
  an $m-j$ dimensional $\crb$-bundle. In particular, $X^\sharp$ is smooth on $\hat M_m$.  
\end{thm}

\begin{proof}
  We already know that if $X$ is an infinitesimal automorphism, then $X^\sharp$, considered as a section of $\holsheaf$, is a
  $\crb$-section. Since $X$ extends microlocally, so does $X^\sharp$, and since $\imag X^\sharp (\theta)=0$, we
  conclude that $\charsheaf \subset \ssheaf (X^\sharp) $. Thus, $\hat \charsheaf \subset \ssheaf ( X^\sharp)$,
  and the theorem follows. 
\end{proof}

As discussed before, the conclusion of Theorem~\ref{thm:infinitesimals} is in a sense optimal without
further assumptions, but does not let us conclude that an infinitesimal automorphism
is actually smooth. 

\begin{cor}\label{cor:nondegenerate}
  If $(M,\crb)$ is an involutive structure and $p\in M$ is a
  regular nondegenerate point, then  any locally bounded  infinitesimal automorphism $X$ near $p$ is smooth
  near $p$. 
\end{cor}

We are now going to discuss a nice sufficient condition for regularity,
starting with the following Lemma.

\begin{lem}\label{lem:simple}
  Assume that $f \in L^{\infty}_{\rm loc} (\Rn)$ satisfies that $x^{\alpha} f (x) = g(x)$ for some $g\in \cinfty(\Rn)$.
  Then $f\in \cinfty(\Rn)$. 
\end{lem}

\begin{proof} The (easy) proof is by induction on $|\alpha|$, with $|\alpha| = 0$ being trivial. Now assume
  that $\alpha = \tilde \alpha + e_j$ for some $|\tilde\alpha|< |\alpha|$. Then $g|_{x_j=0} = 0 $, so
  that we can write $g = x_j \tilde g$ with $\tilde g \in \cinfty (\Rn)$. Thus $x^{\tilde \alpha} f = \tilde g$.
  \end{proof} 

  Based on this Lemma, we introduce a notion of mild singularity of a quotient sheaf $\faktor{ \esheaf}{\ssheaf}$, where
  $\ssheaf\subset\esheaf\subset \Gamma (\cdot, E^*)$ are submodules of the sheaf of sections
  of the dual of a  vector bundle $E$ (we are choosing to think about sections of $E^*$ because it's going to be the setting
  we apply this in; the notion obviously also makes sense for sections of $E$). 
  
   \begin{defn}\label{def:normalcrossingsgen}
    We say that $\ssheaf \subset \esheaf$  are   {\em equal modulo normal crossings} at $p\in M$ if there
    exist coordinates $x \in \R^{n+m}$ for $M$ near p, a neigbourhood $U$ of $p$, 
    sections $\omega_1 , \dots , \omega_{N}  \in \esheaf (U)  $ spanning $\esheaf (U)$, and multiindices $\alpha_1, \dots, \alpha_N
    \in \N^{n+m}$  such that near $p$,  
    the sections $x^{\alpha_1} {\omega_1}, \dots, x^{\alpha_N} {\omega_N}$ are sections
    of $\ssheaf$. 
  \end{defn}
  The idea is  that in this case, the support of $\faktor{\esheaf}{\ssheaf}$ is
  ``not too bad'' with respect to sections of $\esheaf$  so that
  we can apply Lemma~\ref{lem:simple} to conclude regularity:

  \begin{lem}\label{lem:smoothness} Assume that $\ssheaf \subset \esheaf$ are equal modulo normal crossings near $p$. If $\eta$ is a locally bounded section
    of $E$ which  
    is smooth modulo $\ssheaf$ in the sense that $ \omega (\eta) \in \cinfty_p(M)$, for every $\omega \in \ssheaf_p$,  then $\eta$ is smooth modulo
    $\esheaf$, i.e. $\omega(\eta) \in \cinfty_p (M)$ for every $\omega \in \esheaf_p$. 
  \end{lem}

     \begin{proof}
       Let $\eta$ be a locally bounded section of $E$, and write $\omega_j (\eta) = \eta_j$, wehere the $\omega_j$ are generators of $\esheaf$
       as in Definition~\ref{def:normalcrossingsgen}.
       We have
    to show that the $\eta_j$ are smooth. By assumption, we know that
    $\eta$ is smooth modulo $\ssheaf$, and thus in local coordinates, $x^{\alpha_j} \omega (\eta) = x^{\alpha_j} \eta_j$ is smooth, and so
    the result follows from Lemma~\ref{lem:simple}.
  \end{proof}

  We will now define the type of degeneracies and characteristics that we can handle. 

  \begin{defn}\label{def:normalcrossings}
    We say that $M$ has {\em normal crossing exceptional locus} at $p\in M$ if
    $\holsheaf+ \overline\holsheaf \subset \C T^* M$ are equal modulo normal crossings at $p$; we say that
    $M$ has {\em normal crossing degeneracy locus} if $\hat \charsheaf \subset \holsheaf$ are equal modulo normal
    crossings. To be exact, $M$ has normal crossing exceptional locus at $p$ if there
    exist coordinates $x \in \R^{n+m}$ for $M$ near p, a neigbourhood $U$ of $p$, 
    forms $\omega_1 , \dots , \omega_{N}  \in \Gamma(U,\C T^*M)  $ spanning $\C T^*_p M$, and multiindices $\alpha_1, \dots, \alpha_N
    \in \N^N$  such that near $p$,  
    by $x^{\alpha_1} {\omega_1}, \dots, x^{\alpha_N} {\omega_N}$ are sections
    of $\holsheaf+ \overline\holsheaf$, and $M$ has normal crossing degeneracy locus at $p$ if  there
    exist coordinates $x \in \R^{n+m}$ for $M$ near p, a neigbourhood $U$ of $p$, 
    forms $\omega_1 , \dots , \omega_{N}  \in \holsheaf (U)  $ spanning $\holsheaf_p$, and multiindices $\alpha_1, \dots, \alpha_N
    \in \N^N$  such that near $p$,  
    by $x^{\alpha_1} {\omega_1}, \dots, x^{\alpha_N} {\omega_N}$ are sections
    of $\hat \charsheaf$.
  \end{defn}

  The following Lemma is immediate from Lemma~\ref{lem:smoothness}.

  \begin{lem}
    If $M$ has normal crossing exceptional locus at $p$, then it is regular at $p$. 
  \end{lem}

  The proofs of Theorem~\ref{thm:mainautos} and Corollary~\ref{cor:mainautos}  are now immediate, and
  we can sharpen Theorem~\ref{thm:infinitesimals} a bit in the generically nondegenerate case.

  \begin{thm}\label{thm:infinitesimals}
  Assume that $X$ is an infinitesimal automorphism of the involutive structure $(M, \crb)$ and that
  $X$ extends microlocally at $p$. If $M$ has normal crossing degeneracy locus at $p$, then $X^\sharp$ is smooth near $p$. If
  $M$ in addition has normal crossing exceptional locus at $p$ (or is regular at $p$), then $X$ is smooth near $p$. 
\end{thm}

Theorem~\ref{thm:infinitesimals} follows since we know that $\theta (X)$ is smooth for every  $\theta \in \hat\charsheaf_p$; if
$M$ has normal crossing degeneracy locus, then $\omega(X)$ is smooth for every $\omega \in \holsheaf_p$. 
  
  We now discuss how one can check the conditions (regularity and nondegeneracy) in a locally
  integrable structure. Recall from section~\ref{sec:involutive} that if we assume that $M$ is locally
  integrable, $\dim \charforms_p M = d$, then there exists integers  $\nu, \mu$ with $2 \nu + d +\mu = m+n$,
  coordinates 
  $ (x,y,s,t) \in \R^\nu_x \times \R^\nu_y \times \R_s^d \times \R_t^\mu$, and a family of basic solutions of the
  form
  \begin{align*}
  &Z_k(x,y,s,t) = x_k + iy_k, \quad k = 1, \dots, \nu,\\
  &W_k(x,y,s,t) = s_k + i\phi_k (x,y,s,t), \quad k =1, \dots, d,
\end{align*}
for real-valued smooth functions $\phi_k$ satisfying $\phi_k(0) = 0$ and $d\phi_k (0)= 0$.

Thus, $\holforms M$ is spanned by $\{ dZ_1, \dots, dZ_\nu,  dW_1, \dots dW_d \}$, and
$\crb$ is spanned by the vector fields
\[
  L_j = \frac{\partial}{\partial \bar{z}_j} - i \sum_{k=1}^d \frac{\partial \phi_k}{\partial \bar{z}_j} N_k, \quad  j = 1, \dots, \nu; \qquad 
  L_{\nu+ j} = \frac{\partial}{\partial t_j} - i\sum_{k=1}^d \frac{\partial \phi_k}{\partial t_j} N_k, \quad j=1, \dots, \mu,
\]
where $
N_k = \sum_{\ell = 1}^d \big({}^tW_s^{-1}\big)_{k\ell} \frac{\partial}{\partial s_\ell}$ for
$k=1,\dots,d$.

Let us first check under which conditions we have a normal crossing locus. Recall that
\[ dZ_j = d x_j + i dy_j,\quad j = 1,\dots, \nu,\]
\[  dW_k = ds_k + i\sum_\ell \phi_{k,s_\ell } ds_\ell  + i \sum_r \phi_{k,t_r } dt_r + i\sum_j \phi_{k,x_j } dx_j  + i \sum_j \phi_{k,y_j } dy_j. \]
Then $\holsheaf+ \overline{\holsheaf}$ is generated by
\[ d x_j, dy_j, \quad j= 1, \dots, \nu, \qquad ds_k, \quad k =1,\dots, d, \qquad \sum_r \phi_{k,t_r} dt_r , \quad k=1,\dots, d. \]
If $\holsheaf+\overline{\holsheaf}$ is to agree with $\Gamma(\cdot, \C T^*M)$ generically, the
rank of the matrix $\phi_t$ therefore generically needs to be $\mu$. If for any choice $\kappa = (k_1, \dots, k_\mu)$ with  $1\leq k_1 < \dots < k_\mu \leq d $ we  denote
\[ \Delta_\kappa :=
  \begin{vmatrix}
    \phi_{k_1, t_1} &\dots & \phi_{k_1, t_\mu} \\
    \vdots & & \vdots \\
    \phi_{k_\mu, t_1} &\dots & \phi_{k_\mu, t_\mu}
  \end{vmatrix} \]
then we have for example that $\Delta_\kappa dt_r $, where $ r=1,\dots, \mu$, is a section of $\holsheaf+ \overline{\holsheaf}$.
We thus have the following simple checkable criterion:

\begin{lem}
  \label{lem:charnormal} If $M$ is a locally integrable structure, and for some $\kappa$ as
  above we can write $\Delta_\kappa = x^\alpha y^\beta s^\gamma t^\delta \tilde \Delta_\kappa$ with
  $\tilde \Delta_\kappa (0) \neq 0$, then $M$ has normal crossing exceptional locus (and in particular is regular) at $p$. 
\end{lem}

We need a basis for $\charsheaf $ as well. For that,
assume that $\sum \alpha^j dZ_j + \sum \beta^j dW_j$ is real. This yields the following system of equations:
\[
  \begin{aligned}
    \alpha^\ell - \bar \alpha^\ell & =  -i \sum_j (\beta^j + \bar \beta^j)  \phi_{j,x_\ell},
    &
    \ell = 1, \dots, \nu \\
    \alpha^\ell + \bar \alpha^\ell & = - \sum_j (\beta^j + \bar \beta^j) \phi_{j,y_\ell} &
    \ell = 1, \dots, \nu  \\
    \beta^k - \bar \beta^k &= i\sum_j (\beta^j + \bar \beta^j) \phi_{j, s_k}  &
    k = 1, \dots, d \\
    0 & = \sum_j  (\beta^j + \bar \beta^j) \phi_{j, t_r} &
    r = 1, \dots, \mu     
  \end{aligned}
\]
The last line determines the real parts of the $\beta$, from which the $\alpha$ and the imaginary parts of the $\beta$ are determined by the first three lines. If we use matrix notation with $\alpha = (\alpha^1 , \dots, \alpha^\nu)$ and $\beta =(\beta^1 ,\dots, \beta^d)$, then
we can compactly write:
\[ \alpha = -i(\real \beta) \phi_z , \quad \imag \beta = (\real \beta) \phi_s , \quad 0 = (\real \beta) \phi_t.\]
Hence necessarily $\dim \charforms_p = \max(d - \rank \phi_t(p), 0)$, and we will
only have nontrivial characteristic forms if $\rank \phi_t(p):= r_p <d$.

If $(M,\crb)$ happens to be real-analytic, then we can find a basis $\{b_1, \dots b_s \}$ of
real-valued real-analytic vectors $b_j = (b_j^1 , \dots ,b_j^d)$, $j=1,\dots, s$ satisfying $b_ j\phi_t = 0$ by Cartan's theorem B. 

If $M$ is merely smooth, this is a bit more complicated, but if we let
$R := \limsup_{p\to 0} r_p$, then we can find a nice
basis of forms for $\charsheaf$ on the set where $\phi_t$ is of maximal rank $R$
using Cramer's rule as follows. For any choice $\delta = (d_1,\dots,d_R)$, with  $1\leq d_1 < \dots < d_R\leq d $  of $R$ integers and  $\gamma = (g_1, \dots ,g_R)$ with $1\leq g_1 < \dots < g_R \leq \mu$ we write 
\[
\Delta ( \delta, \gamma) = \det( \phi^\delta_{\gamma} ) = 
  \begin{vmatrix}
    \phi_{d_1,t_{g_1}} & \dots & \phi_{d_1,t_{g_R}} \\
    \vdots & \ddots & \vdots \\
\phi_{d_R,t_{g_1}} & \dots & \phi_{d_R,t_{g_R}} \\
  \end{vmatrix}
  , 
\]
and $(\phi^\delta_\gamma)^c$ for the classical adjoint of $\phi^\delta_\gamma $. We
then set if e.g. $\delta = \gamma = (1,\dots, R)$
\[b_{\delta,\gamma,\ell} = (\phi^{\ell}_\delta (\phi^\delta_\gamma)^c, 0, \dots, \overbrace{\Delta(\gamma,\delta)}^{{\ell}\text{th spot}}, \dots, 0), \quad R < \ell \leq \mu\]
with the obvious modifications for other choices of $\delta$. These vectors will necessarily span at each point $p$
of maximal rank of $\phi_t$ the kernel of $\phi_t$ but will typically not be reduced
(i.e. their components will typically contain a lot of common factors). 

In any case, each non-zero vector $b$ of real functions satisfying $b\phi_t= 0$ gives rise to a characteristic form
\[ \theta_b := -\sum_{j=1 }^\nu 2i b \phi_{z_j} dZ^j + \sum_{k=1}^s (b^k + i b \phi_{s_k} ) dW^k =:- 2i b \phi_z dZ + b(I+ i \phi_s) dW = b \left(-2i \phi_z dZ + (I+i\phi_s) dW\right).  \]

We thus have the following Lemma:

\begin{lem}
  Assume that $M$ is locally integrable, given by basic first integrals $Z_k$ and $W_k$ as above. Define the
  spaces $E_k\subset \C^{m} $ by
  \[ E_k := \mathrm{span} \left\{ L^\alpha (-2ib \phi_z, b(I + i \phi_s))|_0 \colon b\phi_t = 0, |\alpha| \leq k    \right\}. \]
  Then $M$ is nondegenerate at $0$ if and only if for some $k$ we have $E_k = \C^m$. 
  In particular, if $M$ is nondegenerate then necessarily
  \[F_k :=\mathrm{span} \left\{ L^\alpha b|_0 \colon b\phi_t = 0, |\alpha| \leq k    \right\}  \]
  satisfies $F_k = \C^d$ for some $k$. 
\end{lem}

We can now generalize a bit our first example. 

\begin{exa}
  Assume that we are given a locally integrable structure defined by first integrals
  \[ W_k = s_k + i t^{\alpha^k} , \quad k = 1,\dots, d  \]
  on $\R^d_s \times \R^{\mu}_t$.
  Then
  \[ \phi_t =
    \begin{pmatrix}
      \alpha^1_1 \dfrac{t^{\alpha^1}}{t_1} & \dots  & \alpha^1_\mu \dfrac{t^{\alpha^1}}{t_\mu}  \\
      \vdots & & \vdots \\
      \alpha^d_1 \dfrac{t^{\alpha^d}}{t_1}  & \dots  & \alpha^d_\mu \dfrac{t^{\alpha^d}}{t_\mu}  \\
    \end{pmatrix}  
  \]
and we know that the rank of the matrix is less than $d$ if and only if
the $\alpha^j$ are linearly dependent (for example, if $\mu < d$ this is always satisfied). The structure has normal crossings exceptional locus if
$\alpha^1 , \dots, \alpha^d$ are not contained in a hyperplane in $\mathbb{Z}^\mu$.

In order to study its nondegeneracy, for any $\lambda = (\lambda_1 , \dots, \lambda_d)$
with $\sum \lambda_k \alpha^k =0$, we can define
\[b_\lambda = \left( \lambda_1 t^{\hat \alpha^1} , \dots , \lambda_d t^{\hat \alpha^d}\right), \quad \hat \alpha^j:= \sum_{k\neq j} \alpha^k.  \]
The vector fields in this case are given by
\[ L_\ell := \dop{t_\ell} - i \sum_{j=1}^d \alpha^j_\ell \frac{t^{\alpha_j}}{t_\ell} \dop{s_j}, \quad \ell = 1,\dots, \mu.  \]
Hence,
\[ \frac{1}{\alpha!} L^\alpha b_\lambda|_{t=0} = (\delta_\alpha^{\hat{\alpha_1}} \lambda_1 , \dots, \delta_\alpha^{\hat{\alpha_d}} \lambda_d), \]
and if the $\alpha^j$ are pairwise different and $\lambda_j \neq 0$ for all $j=1,\dots, d$, the structure will be nondegenerate at $0$. 

\end{exa}

\begin{exa}
  If we consider the structure defined by
  \[ W_1 = s_1 + i t_1^2 , \quad W_2 = s_2 + i t_1 t_2 , \quad W_3 = s_3 + i t_2^2  \]
  has normal crossings and therefore is regular at $0$. It is also nondegenerate at $0$.  
\end{exa}

\begin{exa}
  Assume that $\mu=1$ and $d=2$, and that we have solutions  $Z = x+ iy$ and 
  \[ W_1 = s_1 + i \frac{t^{k+1}}{k+1} |Z|^2 , \quad W_2 = s_2 + i\frac{ t^{\ell+1}}{\ell+1} |Z|^2.  \]
  Then
  \[ \phi_t =
    \begin{pmatrix}
     t^k |Z|^2 \\
      t^\ell |Z|^2
    \end{pmatrix}, \]
  and we see that we have normal crossings.
  We choose $b=(t^\ell, -t^k)$ to check for nondegeneracy and therefore
  have to compute the derivatives of 
  \[ u =  \left(-2i\left(\frac{1}{k+1} - \frac{1}{\ell+1}\right) t^{\ell + k + 1} \bar z, t^{\ell}
      , -t^k\right) \]
  with respect to the vector fields
  \[ L_1 = \dop{\bar z_j} - i z\frac{t^{k+1}}{k+1}   \dop{s_1} -  i z \frac{ t^{\ell+1}}{\ell+1} \dop{s_2}, \qquad L_2 =
    \dop{t} -  i t^k |z|^2 \dop{s_1} -  it^\ell |z|^2 \dop{s_2}.  \] 
  Thus $L_1 L_2^{k+\ell +1} u (0) $, $L_2^\ell u (0)$, and $L_2^k u (0)$ are linearly independent, and
  we have nondegeneracy. 
\end{exa}

%%%%%%%%%%%%%%%%%%%%%%%%%%%%%%%%%%%%%%%%%%%%%%%%%%%%%%%%%%%
%%%%%%%%%%%%%%%%%%%%%%%%%%%%%%%%%%%%%%%%%%%%%%%%%%%%%%%%%%%
%
\section{Microlocal regularity for approximate solutions of a class of non-homogeneous systems}
%
%%%%%%%%%%%%%%%%%%%%%%%%%%%%%%%%%%%%%%%%%%%%%%%%%%%%%%%%%%%
%%%%%%%%%%%%%%%%%%%%%%%%%%%%%%%%%%%%%%%%%%%%%%%%%%%%%%%%%%%
%
In this section we shall prove two results regarding the microlocal regularity of approximate solutions both for a complex vector field and for a class of systems of non-homogeneous equations, namely Theorems \ref{thm:microlocal_approximate_solution} and \ref{thm:microlocal_nonhomogeneous_system}.

Let $U \subset \mathbb{R}^{N+1}$ be an open neighborhood of the origin, and consider the following complex vector field:
\begin{equation}
	\mathrm{L} = \frac{\partial}{\partial t} + i\sum_{j = 1}^N b_j(x,t) \frac{\partial}{\partial x_j} =: \frac{\partial}{\partial t} + i b(x,t) \cdot \dop{x},
\end{equation}
where the coefficients $b_j$ are real-valued, smooth function on $U$, and we
introduce vector notation to be used below. 
We shall also consider the complex vector field $L_1$ on 
$U \times \mathbb{R}$ defined by
\begin{equation}
	\mathrm{L}_1 = \frac{\partial}{\partial s} + i \mathrm{L} = \dop{s} + i \frac{\partial}{\partial t}  - \sum_{j = 1}^N b_j(x,t) \frac{\partial}{\partial x_j}
	= 2 \dop{\bar w} - \sum_{j = 1}^N b_j(x,t) \frac{\partial}{\partial x_j}.
\end{equation}
More generally, given any continuous linear operator $A : \mathcal{C}^\infty(U; \mathbb{C}^r) \longrightarrow \mathcal{C}^\infty(U;\mathcal{C}^r)$, we shall denote by $A_1$ the operator acting on $\mathcal{C}^\infty(U \times \mathbb{R}; \mathbb{C}^r)$ given by 
\begin{equation*}
	A_1 =: \frac{\partial}{\partial s} + i A
\end{equation*}
\begin{defn}
Let $A$ be a continuous linear operator $A : \mathcal{C}^\infty(U; \mathbb{C}^r) \longrightarrow \mathcal{C}^\infty(U; \mathbb{C}^r)$. We say that a $\mathcal{C}^\infty$-smooth vector-valued function $u(x,t,s)$ on $U \times \mathbb{R}$ is an $s$-approximate solution for $A_1$ if for every compact set $K \subset U$ and $k \in \mathbb{Z}_+$, there exists a positive constant $C(k) > 0$ such that
	\begin{equation*}
		\| A_1 u(x,t,s)\| \leq C(k)|s|^k ,
		\eqno{\forall (x,t) \in K, \forall |s| \ll 1}
	\end{equation*}
\end{defn}
\begin{rmk}
When $A$ extends as a continuous map $A : \mathcal{C}^k(U; \mathbb{C}^r) \longrightarrow C^0(U; \mathbb{C}^r)$, for some $k \in \mathbb{N}$, then we can consider $\mathcal{C}^k$ approximate solutions for $A_1$.
\end{rmk}

%NOTE: I'm adding in something here which is really not necessary, but useful somehow to connect with Theorem 1.5? 

In order to prepare for our discussion of Theorem~\ref{thm:microlocal_approximate_solution} let us recall that  $s$-approximate solutions are easily constructed for given smooth data 
 $u(x,t,0)$. This is due to the fact
that solving the (homogeneous) equation $A_1 u = 0$ is straightforward in the space 
$\fpstwo{\cinfty(\Omega)}{s}$, being equivalent to solving the Cauchy problem  $\partial_s \hat u = -i A \hat u$, $\hat u (x,t,0) = u_0 (x,t)$ for
the formal power series expansion in $s$ of $u(x,t,s)$, i.e. 
\[ \hat u(x,t,s) = \sum_{k=0}^\infty \frac{\partial_s^k u (x,t,0)}{k!} s^k \in \fpstwo{\cinfty(U)}{s}. \]
This yields formally  
\[  \hat u (x,t,s) = e^{-i s \ourL} \hat u (x,t,0) = \sum_{k=0}^\infty \frac{ (-i A)^{k}  u_0}{k!} s^k.  \]
Choosing a cutoff function $\chi(s)$ and an appropriate rapidly increasing sequence $R_k$
provides an $s$-approximate solution   
\begin{equation}
\label{e:formalsoln} u(x,t,s) = \sum_{k=0}^\infty  \frac{ (-i A)^{k} u_0(x,t)}{k!} \chi \left( R_ks \right) s^k;
\end{equation}
%
%we will discuss the necessary details in our proof of Theorem~\ref{thm:vectorapproximate} below.
%
In order to make the present paper self-contained, we shall present a proof of the claim above.
\begin{lem}
For every $V \Subset U$, there exists a rapidly increasing sequence $\{R_k\}_{k \in \mathbb{N}}$ such that the function $u(x,t,s)$ giving by \eqref{e:formalsoln} is $\mathcal{C}^\infty$-smooth on $V \times \mathbb{R}$ with compact support on the $s$-variable, $u(x,t,0) = u_0(x,t)$ for every $(x,t) \in V$, and $u$ is a $s$-approximate solution for $A_1$. 
\end{lem}
\begin{proof}
Let $V \Subset U$. For every $(\alpha, l, m, k) \in \mathbb{Z}_+^N \times \mathbb{Z}_+ \times \mathbb{Z}_+ \times \mathbb{Z}_+$ we set
\begin{equation}
	C(\alpha, l, m, k) =: \sup_{\substack{(x,t) \in \overline{V} \\ s \in [-1,1]}} \sum_{q = 0}^m \binom{m}{q} \frac{|\chi^{(q)}(s)|}{(k-m+q)!} \|\partial_x^\alpha \partial_t^l(-iA)^k u_0(x,t) \|.
\end{equation}
Note that the quantity above is well defined since $A$ is continuous on $\mathcal{C}^\infty(U)$. Choose $R_k$ such that for every $k$, 
\begin{equation}
	\frac{\sup_{|\alpha|+l+m \leq k}C(\alpha, l, m, k)}{R_k} \leq \frac{1}{2^k}.
\end{equation}
For every $n \in \mathbb{N}$ we denote by $u_n$ the following partial sum
\begin{equation*}
	u_n(x,t,s) =: \sum_{k = 0}^n \frac{(-iA)^ku_0(x,t)}{k!}\chi(R_k s)s^k.
\end{equation*}
We claim that the sequence $u_n$ is Cauchy in $\mathcal{C}^\infty(V\times]-1,1[)$. Indeed, let $\kappa \leq n \leq  n^\prime \in \mathbb{Z}_+$, and choose $(\alpha, l, m) \in \mathbb{Z}_+^N \times \mathbb{Z}_+ \times \mathbb{Z}_+$ with $|\alpha| + l + m \leq \kappa $.
Then
\begin{align*}
	\partial_x^\alpha \partial_t^l \partial_s^m (u_{n+n^\prime}&(x,t,s) - u_n(x,t,s))  = \\
	&= \sum_{k=n}^{n + n^\prime} \frac{1}{k!} \bigg(\sum_{q=0}^m \binom{m}{q} R_k^q \chi^{(q)}(R_k s) \frac{k!}{(k-m+q)!} s^{k-m+q} \bigg) \partial_x^\alpha \partial_t^l (-iA)^k e^j \\
	& = \sum_{k=n}^{n + n^\prime} s^{k-m} \bigg(\sum_{q=0}^m \binom{m}{q} \frac{(R_k s)^q \chi^{(q)}(R_k s)}{(k-m+q)!} \bigg) \partial_x^\alpha \partial_t^l (-iA)^k e^j.
\end{align*}	 
Now we notice that $\chi(R_k s) = 0$, if $|s| \geq 1/|R_k|$, so we have that for $(x,t) \in V$ and $s \in \mathbb{R}$,
\begin{align*}
	\big\| \partial_x^\alpha \partial_t^l \partial_s^m (u_{n+n^\prime}&(x,t,s) - u_n(x,t,s)) \big\| \leq \\
	& \leq \sum_{k=n}^{n + n^\prime} |s|^{k-m} \bigg(\sum_{q=0}^m \binom{m}{q} \frac{|R_k s|^q |\chi^{(q)}(R_k s)|}{(k-m+q)!} \bigg) \|\partial_x^\alpha \partial_t^l (-iD)^k e^j \| \\
	& \leq \sum_{k=n}^{n + n^\prime} \left(\frac{1}{R_k} \right)^{k-m} C(\alpha, l, m, k) \\
	& \leq \sum_{k=n}^{n + n^\prime} \frac{\sup_{|\alpha^\prime|+l^\prime + m^\prime \leq k}C(\alpha^\prime, l^\prime, m^\prime, k)}{R_k}  \\
	& \leq \sum_{k=n}^{n + n^\prime} \frac{1}{2^k}.
\end{align*}	 
We then conclude that $u_n$ converges in the $\mathcal{C}^\infty$ topology on $V\times \mathbb{R}$ to a $\mathcal{C}^\infty$-smooth (vector-valued) function $u$. To prove that $u$ is an $s$-approximate solution for $A_1$ we start writing
\begin{align*}
	A_1 u(x,t,s) &= \sum_{k=0}^\infty \frac{ R_k \chi^\prime (R_k s)s^{k}}{k!}(-iA)^k u_0(x,t) + \sum_{k=1}^\infty \frac{\chi^\prime(R_k s)k s^{k-1}}{k!} (-iA)^k e^j \\
	& - \sum_{k=0}^\infty \frac{\chi(R_k s) s^k}{k!}(-iA)^{k+1} u_0(x,t) \\
	& = \sum_{k=0}^\infty \frac{s^k}{k!}G_k(x,t,s),
\end{align*}
where $G_k(x,t,s) = R_k \chi^\prime(R_k s)(-iA)^k u_0(x,t) + (\chi(R_{k+1} s) - \chi(R_k s))(-iA)^{k+1} u_0(x,t)$, and in particular we have that $\partial_s^l G_k(x,t,0) = 0$, for every $l \in \mathbb{Z}_+$. Thus $\partial_s^\ell \mathrm{D}_1 {\omega}^j\big|_{s=0} = 0$, for every $\ell \in \mathbb{Z}_+$, and this means that $u$ is an $s$-approximate solution for $A_1$.
 \end{proof}
Let us go back to the vector field $\mathrm{L}$. In view of the previous lemma, fix $V \Subset U$ an open neighborhood of the origin, and let $Z_1(x,t,s), \dots, Z_N(x,t,s)$ be $\mathcal{C}^\infty$-smooth $s$-approximate solutions for $\mathrm{L}_1$ on $V$, with $Z_j(x,t,0) = x_j$, for $j = 1, \dots, N$. Then we can write $Z(x,t,s) = x + s\psi(x,t,s)$, where $\psi = (\psi_1, \dots, \psi_N)$ is a $\mathcal{C}^\infty$-smooth map. Now we shall enunciate a general and useful fact relating $\psi$ with the coefficients $b_j$s of $\mathrm{L}$.
\begin{prop}
	For every $l \in \mathbb{Z}_+$ and $j = 1, \dots, N$ we have that 
	\begin{equation}\label{eq:relation_psi_b}
		\partial_s^l \psi_j|_{s=0} = \frac{(-i\mathrm{L})^l b_j}{l+1}.
	\end{equation}	
\end{prop}
\begin{proof}
%We observe that $-i \ourL x_j = b_j$ and therefore the formula follows by the
%observation above, in particular \eqref{e:formalsoln}, applied to $u(x,t,0) = x_j$. 
%
Let $j \in \{1, \dots, N\}$. Since $Z_j$ is an $s$-approximate solution for $\mathrm{L}_1$ we have that for every $l \in \mathbb{Z}_+$,
\begin{equation*}
	\partial_s^l \mathrm{L}_1 Z_j|_{s=0} = 0.
\end{equation*}
Since 
\begin{align*}
	\mathrm{L}_1 Z_j(x,t,s) & = (\partial_s +i\mathrm{L})(x_j + s\psi_j(x,t,s))
		\\
 	& = \psi_j(x,t,s) - b_j(x,t) + s\mathrm{L}_1\psi_j(x,t,s),
\end{align*}
 we have that $\psi_j|_{s = 0} = b_j$. Now we shall proceed by induction on $l$. So suppose that \eqref{eq:relation_psi_b} is valid for $l = 0, 1, \dots, l_0$. We notice that since the coefficients of $\mathrm{L}_1$ does not depend on $s$,
 \begin{align*}
 	\partial_s^{l_0+1} \mathrm{L}_1 Z_j(x,t,s) & = \partial_s^{l_0+1} \big(\psi_j(x,t,s) - b_j(x,t) + s\mathrm{L}_1\psi_j(x,t,s)\big) \\
 	& = \partial_s^{l_0+1} \psi_j(x,t,s) + \sum_{k = 0}^{l_0 + 1} \binom{l_0 + 1}{k} (\partial_s^{l_0 +1 - k} s) \mathrm{L}_1 \partial_s^k \psi_j(x,t,s) \\
 	& =  \partial_s^{l_0+1} \psi_j(x,t,s) + (l_0 + 1)\mathrm{L}_1 \partial_s^{l_0} \psi_j(x,t,s) + s \mathrm{L}_1 \partial_s^{l_0 +1} \psi_j(x,t,s) \\
 	& = (l_0 + 2)\partial_s^{l_0+1} \psi_j(x,t,s) + (l_0 + 1)i\mathrm{L} \partial_s^{l_0} \psi_j(x,t,s) + s \mathrm{L}_1 \partial_s^{l_0 +1} \psi_j(x,t,s)
 \end{align*}
Thus evaluating the equality above on $s = 0$ and using the induction hypothesis we obtain
 \begin{align*}
 	\partial_s^{l_0+1} \psi_j|_{s = 0} & = -\frac{(l_0 + 1)}{l_0 + 2}i\mathrm{L} \partial_s^{l_0} \psi_j|_{s = 0} \\
 	& =  \frac{(-i\mathrm{L})^{l_0+1} b_j}{l_0+2}
 \end{align*}
\end{proof}
\begin{rmk}
	Suppose that $\Im i\mathrm{L}b(x,t) = 0$, for every $(x,t)$ in some open neighborhood of the origin. Then near the origin $b(x,t) = b(x)$. Since on functions independent of $t$ the vector field $\mathrm{L}$ is purely imaginary,  $i \mathrm{L}$ becomes a real vector field on such functions, and therefore $\Im (i\mathrm{L})^k b \equiv 0$ for every $k \in \mathbb{N}$.
\end{rmk} 
%
%In order to facilitate the notation of the next results, we shall denote by $\mathcal{D}^\prime(V)\otimes_\ast\mathcal{C}^\infty(V\times]-\delta,\delta[)$ the pullback of $\mathcal{D}^\prime(V)\otimes\mathcal{C}^\infty(V\times]-\delta,\delta[)$ by the map $\mathfrak{i}: V \times ]-\delta, \delta[ \longrightarrow V \times V \times ]-\delta,\delta[ $ giving by $\mathfrak{i}(x,t,s) = (x,t,x,t,s)$, \textit{i.e.} the elements in $\mathcal{D}^\prime(V)\otimes_\ast\mathcal{C}^\infty(V\times]-\delta,\delta[)$ can be written as
%
%\begin{equation*}
%	\sum_{j = 1}^n u_j(x,t) v_j(x,t,s),
%\end{equation*}
%
%for some $u_1, \dots, u_n \in \mathcal{D}^\prime(V)$ and $v_1, \dots, v_n \in \mathcal{C}^\infty(V \times ]-\delta,\delta[)$. We shall also denote by $\mathcal{C}^\infty_{\text{flat}}(V\times]-\delta,\delta[)$ as the set of all $\mathcal{C}^\infty$-smooth functions on $V \times ]-\delta, \delta[$ that are flat at $s=0$.
%
For every open interval $I\subset \mathbb{R}$ containing the origin, we shall denote by $\mathcal{C}^\infty_{\text{flat}}(I; \mathcal{D}^\prime(V))$ the space of all $u \in \mathcal{C}^\infty(I; \mathcal{D}^\prime)$ such that, there exists $\delta > 0$ for which for every compact set $K \subset V$ there exists $m \in \mathbb{N}$ satisfying the following: for every $k \in \mathbb{N}$ there is a constant $C = C_{K,k} > 0$ such that
\[
	|\langle u(s), \phi(x,t) \rangle| \leq C_{K,k} |s|^k \sum_{|\alpha| \leq m} \sup_{K} |\partial^\alpha \phi|, \qquad \forall \phi \in \mathcal{C}^\infty_c(K), \forall |s| \leq \delta. 
\]	
\begin{thm}\label{thm:microlocal_approximate_solution}
	Let $\xi^0 \in \mathbb{R}^N \setminus \{ 0 \}$, and assume that $\frac{\partial b}{\partial t}(0) \cdot \xi^0 > 0$. Let $V\subset U$ be an open neighborhood of the origin, $\delta>0$, and $u \in \mathcal{C}^\infty(]-\delta, \delta[; \mathcal{D}^\prime(V))$ be such that $\mathrm{L}_1(u) \in \mathcal{C}_\text{flat}^\infty(]-\delta, \delta[; \mathcal{D}^\prime(V))$. Then $(\xi^0, 0) \notin \mathrm{WF} (u_0)$, where $u_0 = u|_{s=0}$.
\end{thm}
%
%\begin{rmk}
%Note that $u|_0$ is well defined since $u = \sum_{\ell = 1}^{k_0} u_\ell h_\ell$, for some $u_\ell \in \mathcal{D}^\prime(V)$ and $h_\ell \in \mathcal{C}^\infty(V\times]-\delta,\delta[)$, so $u_0 = \sum_{\ell = 1}^{k_0} u_\ell h_\ell|_{s=0}$.
%\end{rmk}
%
\begin{proof}
	It is enough to prove that for some $\chi \in \mathcal{C}^\infty_c(V)$, satisfying $\chi \equiv 1$ in some neighborhood of the origin, some open cone $\Gamma \subset \mathbb{R}^{N+1} \setminus 0$ with $(\xi^0, 0) \in \Gamma$, some constants $\kappa > 0$ and $C_k > 0$, for $k \in \mathbb{Z}_+$, and some neighborhood of the origin $V_0 \subset V$ , the following estimate holds true:
\begin{equation*}
	|\mathfrak{F}_\kappa[\chi u_0](x,t;\xi,\tau)| \leq \frac{C_k}{|(\xi,\tau)|^k},
	\eqno{\forall (x,t) \in V_0, (\xi,\tau) \in \Gamma, |(\xi,\tau)| \geq 1, k \in \mathbb{Z}_+ ,}
\end{equation*} 
where
\begin{align*}
	\mathfrak{F}_\kappa[\chi u_0](x,t;\xi,\tau) = \left\langle \chi(x^\prime, t^\prime)u(0,x^\prime, t^\prime), e^{i (\xi,\tau) \cdot (x - x^\prime, t - t^\prime) - \kappa |(\xi, \tau)||(x - x^\prime, t - t^\prime)|^2}  \right\rangle.
\end{align*}
So let $V_0 \Subset V_1 \Subset V_2 \Subset V$ be open neighborhoods of the origin as small as we want, and let $\chi \in \mathcal{C}_c^\infty(V)$ be such that $\chi \equiv 1$ in $V_1$. Consider $Z_j(x,t,s) = x_j + s\psi_j(x,t,s)$, $j = 1, \dots, N$ as before, and set $Z_{N+1}(x,t,s) = t - is$. We can assume that $\mathrm{d}Z_1 \wedge \cdots \wedge \mathrm{d} Z_{N+1} \neq 0$ in $V_ 2\times \mathbb{R}$.
Thus $\{\mathrm{d}Z_1, \dots, \mathrm{d}Z_{N+1}\}$ defines a locally integrable structure on $V_2 \times [-\delta, \delta]$.  
 
Let $\{\mathrm{M}_1, \dots, \mathrm{M}_{N+1}\}$ be the dual vector fields of the family $\{\mathrm{d}Z_1, \dots, \mathrm{d}Z_{N+1}\}$, so the vector field
\begin{align*}
	\mathrm{L}^\sharp = \mathrm{L}_1 - \sum_{j = 1}^{N + 1} \mathrm{L}_1(Z_j) \mathrm{M}_j,
\end{align*}
satisfies $\mathrm{L}^\sharp Z_j = 0$, for $j = 1, \dots, N + 1$, and $\mathrm{d}s (\mathrm{L}^\sharp) = 1$. Since $\mathrm{L}_1 Z_j$ vanishes to infinite order in $s$ at $s = 0$, we also have that $\mathrm{L}^\sharp (u) \in \mathcal{C}_\text{flat}^\infty(]-\delta, \delta[; \mathcal{D}^\prime(V))$. % Let $k_0 \in \mathbb{N}$, $u_\ell \in \mathcal{D}^\prime(V)$ and $h_\ell \in \mathcal{C}^\infty(V\times]-\delta,\delta[)$, $\ell = 1, \dots, k_0$, be such that $u = \sum_{\ell = 1}^{k_0} u_\ell h_\ell$. For each $\ell = 1, \dots, k_0$, take $\{u_\ell^\nu\}_{\nu \in \mathbb{N}}\subset\mathcal{C}^\infty(V)$ such that $u_\ell^\nu \longrightarrow u_\ell$ in $\mathcal{D}^\prime(V_2)$ for $\ell = 1, \dots, k_0$. 
Since $\mathcal{C}^\infty([-\delta, \delta[) \otimes \mathcal{D}^\prime(V)$ is dense in $\mathcal{C}^\infty(]-\delta, \delta[; \mathcal{D}^\prime(V))$, there exists $\{u_\nu\}_\nu \subset \mathcal{C}^\infty(V \times ]-\delta,[)$ a sequence of smooth functions such that $u_\nu \longleftarrow u$ in $\mathcal{C}^\infty(]-\delta, \delta[; \mathcal{D}^\prime(V))$, and 
\[
	\iint_V u_\nu(x,t,s) v(x,t) \mathbb{d} x \mathrm{d}t \longrightarrow \left\langle u(s), v \right\rangle, \qquad  \forall v \in \mathcal{C}_c^\infty(V), \,\forall s \in ]-\delta, \delta[.
\]
Writing $Z = (Z_1, \dots , Z_{N+1})$, we define
\begin{align*}
	Q(x,t,x^\prime, t^\prime, \xi,\tau, s) & =  i (\xi,\tau) \cdot ((x,t) - Z(x^\prime, t^\prime,s)) - \kappa|(\xi, \tau)|\langle (x,t) - Z(x^\prime, t^\prime, s)\rangle^2,
\end{align*}
where $\langle Z \rangle = \sum_j Z_j^2$ and $\kappa$ is a positive constant to be chosen later.  We also set, for every $\nu \in \mathbb{N}$,
\begin{align*}
	q_\nu(x,t,x^\prime, t^\prime, \xi,\tau, s) = \chi(x^\prime,t^\prime) u_\nu(x^\prime, t^\prime,s)e^{Q(x,t,x^\prime, t^\prime, \xi,\tau, s)}\mathrm{d}Z(x^\prime,t^\prime, s),
\end{align*}
where $\mathrm{d}Z = \mathrm{d}Z_1 \wedge \dots \wedge \mathrm{d}Z_{N+1}$. Since $e^Q$ is a holomorphic function in the $Z$, and we know that $L^\sharp (Z_j) = 0$ and $\ourd s (L^\sharp) =1$, we have that
\begin{align*}
	\mathrm{d}_{(x^\prime, t^\prime, s)} q_\nu(x,t,x^\prime, t^\prime, \xi,\tau, s) & = \mathrm{L}^\sharp(\chi u_\nu)(x^\prime, t^\prime, s)e^{Q(x,t,x^\prime, t^\prime, \xi,\tau, s)} \mathrm{d} s \wedge \mathrm{d}Z(x^\prime,t^\prime, s),
\end{align*}
and therefore by Stokes' Theorem we have that, for any $0 < \gamma \leq \delta$,
\begin{align}\label{eq:Stokes1}
	\int_{V_2 \times [0, \gamma]}\mathrm{L}^\sharp(\chi u_\nu)(x^\prime, t^\prime, s) & e^{Q(x,t,x^\prime, t^\prime, \xi,\tau, s)}  \mathrm{d} s \wedge \mathrm{d}Z(x^\prime,t^\prime, s) = \\ \nonumber
	& = \int_{\partial(V_2 \times [0, \gamma])}(\chi u_\nu)(x^\prime, t^\prime, s)e^{Q(x,t,x^\prime, t^\prime, \xi,\tau, s)}  \mathrm{d}Z(x^\prime,t^\prime, s) \\ \nonumber
	& = \int_{V_2}(\chi u_\nu)(x^\prime, t^\prime, \gamma)e^{Q(x,t,x^\prime, t^\prime, \xi,\tau, \gamma)}  \mathrm{d}Z(x^\prime,t^\prime, \gamma) \\ \nonumber
	& - \mathfrak{F}_\kappa[\chi u_\nu|_{s=0}](x^\prime, t^\prime;\xi,\tau),
\end{align}
in other words,
\begin{align}\label{eq:Stokes_nu}
	\mathfrak{F}_\kappa[\chi u_\nu|_{s=0}](x^\prime, t^\prime;\xi,\tau)  &=  \int_{V_2}(\chi u_\nu)(x^\prime, t^\prime, \gamma)e^{Q(x,t,x^\prime, t^\prime, \xi,\tau, \gamma)}  \det Z_{x,t}(x^\prime,t^\prime, \gamma) \mathrm{d} x^\prime \mathrm{d} t^\prime \\ \nonumber
	& - \int_{[0, \gamma]} \iint_{V_2 }u_\nu(x^\prime,t^\prime,s)\mathrm{L}^\sharp(\chi)(x^\prime, t^\prime, s)  e^{Q(x,t,x^\prime, t^\prime, \xi,\tau, s)}   \det Z_{x,t}(x^\prime,t^\prime, s) \mathrm{d} x^\prime \mathrm{d} t^\prime \mathrm{d} s  \\ \nonumber
	& - \int_{[0, \gamma]} \iint_{V_2 }\chi(x^\prime,t^\prime)\mathrm{L}^\sharp(u_\nu)(x^\prime, t^\prime, s)  e^{Q(x,t,x^\prime, t^\prime, \xi,\tau, s)}   \det Z_{x,t}(x^\prime,t^\prime, s) \mathrm{d} x^\prime \mathrm{d} t^\prime \mathrm{d} s .
\end{align}
Taking $\nu \rightarrow \infty$ on both sides of \eqref{eq:Stokes_nu} we have that
\begin{align}\label{eq:Stokes2}
	\mathfrak{F}_\kappa[\chi u_0](x^\prime, t^\prime;\xi,\tau) & = \left\langle (\chi u)(x^\prime, t^\prime, \gamma), e^{Q(x,t,x^\prime, t^\prime, \xi,\tau, \gamma)} \det Z_{x,t}(x^\prime,t^\prime, \gamma) \right\rangle \\ \nonumber
	& - \int_{[0, \gamma]} \bigg\langle \big(u\mathrm{L}^\sharp(\chi)\big)(x^\prime, t^\prime, s),  e^{Q(x,t,x^\prime, t^\prime, \xi,\tau, s)}  \det Z_{x,t}(x^\prime,t^\prime, s) \bigg\rangle \mathrm{d} s = \\ \nonumber
	& - \int_{[0, \gamma]} \bigg\langle \big(\chi\mathrm{L}^\sharp(u)\big)(x^\prime, t^\prime, s),  e^{Q(x,t,x^\prime, t^\prime, \xi,\tau, s)}  \det Z_{x,t}(x^\prime,t^\prime, s) \bigg\rangle \mathrm{d} s. 
\end{align}
In view of \eqref{eq:relation_psi_b}, since the $b_j$s are real we have that $\Im \psi_j(x,t,0) = 0$, and $\Im \partial_s \psi_j(x,t,0) = - \partial_t b_j(x,t)/2$, for $j = 1, \dots, N$. Thus, for $j = 1, \dots, N$,
\begin{align*}
	\Im \psi_j(x^\prime, t^\prime, s^\prime) & = \Im \psi_j(x^\prime, t^\prime, 0) + \Im \partial_t \psi_j(x^\prime, t^\prime, 0) s^\prime + O((s^\prime)^2) \\
	& = -\frac{1}{2}\frac{\partial b_j}{\partial t}(x^\prime, t^\prime) s^\prime + O((s^\prime)^2).
\end{align*}
Now since $\frac{\partial b}{\partial t}(0) \cdot \xi^0 > 0$, choosing $V_2$ small enough,
we can assume that there exists $\rho >0$ such that
$\frac{\partial b}{\partial t}(x^\prime, t^\prime) \cdot \xi^0 > \rho |\xi^0|$ 
for every $(x^\prime, t^\prime) \in V_2$. 
%$\rho > 0$, which is going to be independent of any particular (small) $r$. 
Choose an open cone $\Gamma^\prime \subset \mathbb{R}^N \setminus 0$  containing $\xi^0$ such that $\frac{\partial b}{\partial t}(x^\prime, t^\prime) \cdot \xi > \frac{\rho}{2} |\xi|$ for every $(x,t)\in V_2$ and $\xi \in \Gamma^\prime$. Therefore, for some $\gamma > 0$,
\begin{equation*}
	\Im \psi(x^\prime, t^\prime, s) \cdot \xi \leq - \frac{\rho}{8} s |\xi|, \qquad \forall (x^\prime, t^\prime) \in V_2, \forall \xi \in \Gamma^\prime, 0 \leq s \leq \gamma.
\end{equation*}
So let $\Gamma \subset \mathbb{R}^{N + 1} \setminus 0$ be the open cone given by
\begin{align*}
	\Gamma = \{ (\xi, \tau) \in \mathbb{R}^{N + 1} \setminus 0 \; : \; \xi \in \Gamma^\prime, |\tau| < \frac{1}{\sqrt{2}} |\xi| \}.
\end{align*}
Note that if $(\xi,\tau) \in \Gamma$, then $|\xi|^2 \leq |(\xi,\tau)|^2 \leq 3/2|\xi|^2$.
Then for $(x,t) \in V_0$, $(x^\prime, t^\prime) \in V_2$, $(\xi, \tau) \in \Gamma$, with $\tau > 0$ (for the case where $\tau$ is negative we use the Stokes' Theorem with $[-\gamma, 0])$ in place of $[0, \gamma]$), and $0 \leq s \leq \gamma$, we can bound the real part of $Q(x,t, x^\prime, t^\prime, \xi, \tau, s)$ as follows: writing $Z^\prime = (Z_1, \dots, Z_N)$,
\begin{align*}
	\Re Q(x,t, x^\prime, t^\prime, \xi, \tau, s) & = \xi \cdot \Im Z^\prime(x^\prime, t^\prime, s) - \tau s \\
	& - \kappa|(\xi,\tau)|\Big(|x - \Re Z^\prime(x^\prime, t^\prime,s)|^2 + |t - t^\prime|^2 - |\Im Z^\prime(x^\prime, t^\prime, s)|^2 - s^2\Big) \\
	& = s \Im \psi(x^\prime, t^\prime, s) \cdot \xi - \tau s - \kappa|(\xi, \tau)|\Big( |x - x^\prime - s \Re \psi(x^\prime, t^\prime, s)|^2 + |t - t^\prime|^2 \Big) \\
	& + \kappa|(\xi,\tau)|\Big(|\Im Z^\prime(x^\prime, t^\prime, s)|^2 + s^2\Big) \\
	& \leq - \frac{\rho}{8}{s}^2 |\xi|  - \kappa|\xi|\Big( |x - x^\prime - s \Re \psi(x^\prime, t^\prime, s)|^2 + |t - t^\prime|^2 \Big) \\ 
	& + \frac{3 \kappa}{2}|\xi| s^2 \Big(|\Im \psi(x^\prime, t^\prime, s)| + 1 \Big).
\end{align*}
%
%We shall shrink $r$ and $\gamma$ for the last time: we shall assume that $|s^\prime| |\Re\psi(x^\prime, t^\prime, s^\prime)| < r/4$, for every $|s^\prime| \leq \gamma$, and $(x^\prime, t^\prime) \in B_{2r}(0)$.
Now we choose $\kappa > 0$ such that
\begin{align*}
	\frac{3\kappa}{2} \left(1 + \sup_{V_2 \times [-\delta, \delta]}|\Im \psi| \right)  < \rho / 16,
\end{align*}
then 
\begin{equation}\label{eq:bound_real_Q_commutaror_lenght_2}
	\Re Q(x,t, x^\prime, t^\prime, \xi, \tau, s)  \leq - \frac{\rho}{16}{s^\prime}^2 |\xi|  - \kappa|\xi|\Big( |x - x^\prime - s \Re \psi(x^\prime, t^\prime, s)|^2 + |t - t^\prime|^2 \Big) . 
\end{equation}
Now since $V_2 \Subset V$, there exist $C^\bullet > 0$ and $k_1>0$ such that, for every $0\leq s\leq \gamma$,
\begin{equation}\label{eq:distribuional_bound_u}
	\left|\left\langle \chi(x^\prime,t^\prime) u(x^\prime, t^\prime, s), \Psi(x^\prime, t^\prime, s) \right\rangle\right| \leq C^\bullet \sum_{|\alpha|\leq k_1} \sup_{(x^\prime,t^\prime) \in \overline{U_2}}|\partial_{x,t}^\alpha|\Psi(x^\prime, t^\prime, s)|,
\end{equation}
for all $\Psi\in \mathcal{C}^\infty_c(\overline{V_2})$. Also note that there exists $\epsilon_1 > 0$ such that, for every $(x,t) \in V_0$, and $(x^\prime, t^\prime) \in V_2\setminus V_1$,  $|x-x^\prime| + |t-t^\prime|) > 3\epsilon_1$. So shrinking $0 < \gamma$ if necessary, we can assume that $\epsilon_1 > \gamma\sup_{V_2\times]-\delta,\delta[}|\psi|$. Thus
\[
	|x-x^\prime - s \Re \psi(x^\prime, t^\prime, s)|^2 + |t-t^\prime|^2 \geq \epsilon_1^2,
\]
for every $(x,t) \in V_0$, and $(x^\prime, t^\prime) \in V_2\setminus V_1$. In view of \eqref{eq:bound_real_Q_commutaror_lenght_2} and \eqref{eq:distribuional_bound_u} we have that
\begin{align*}
	\left|\left\langle (\chi u)(x^\prime, t^\prime, \gamma), e^{Q(x,t,x^\prime, t^\prime, \xi,\tau, \gamma)} \det Z_{x,t}(x^\prime,t^\prime, \gamma) \right\rangle \right| & \leq \text{Const} \cdot (1+|\xi|)^{k_1} e^{-\frac{\rho\gamma^2}{16}|\xi|} \\
	& \leq \text{Const} \cdot e^{-\frac{\rho\gamma^2}{32}|\xi|}
\end{align*}
and
\begin{align*}
	\left|\int_{[0, \gamma]} \bigg\langle \big(u\mathrm{L}^\sharp(\chi)\big)(x^\prime, t^\prime, s),  e^{Q(x,t,x^\prime, t^\prime, \xi,\tau, s)}  \det Z_{x,t}(x^\prime,t^\prime, s) \bigg\rangle \mathrm{d} s \right| &  \leq \text{Const} \cdot  (1+|\xi|)^{k_1} e^{-\kappa \epsilon_1^2|\xi|} \\
	& \leq \text{Const} \cdot e^{-\frac{\kappa \epsilon_1^2}{2}|\xi|},
\end{align*}
for $\mathrm{L}^\sharp \chi(x^\prime, t^\prime) = 0$ if $(x^\prime, t^\prime) \in V_1$ and we are taking $(x,t) \in V_0$. Now recall that $\mathrm{L}^\sharp u \in \mathcal{C}^\infty_\text{flat}(]-\delta,\delta[;\mathcal{D}^\prime(V))$, therefore, there exists $m_1 \in \mathbb{N}$ such that for every $k \in \mathbb{Z}_+$, there is a constant $C_k > 0$ such that, assuming $\gamma$ small enough,
\begin{align}\label{eq:Lsharp_u}
	\bigg|\bigg\langle \big(\chi\mathrm{L}^\sharp(u)\big)(x^\prime, t^\prime, s), &  e^{Q(x,t,x^\prime, t^\prime, \xi,\tau, s)}  \det Z_{x,t}(x^\prime,t^\prime, s) \bigg\rangle \bigg| \leq \\ \nonumber
	& \leq C_k s^k \sum_{|\alpha|\leq m_1} \sup_{(x^\prime,t^\prime)\in V_2} \bigg| \partial_{x^\prime,t^\prime}^\alpha \chi(x^\prime,t^\prime)e^{Q(x,t,x^\prime, t^\prime, \xi,\tau, s)}  \det Z_{x,t}(x^\prime,t^\prime, s) \bigg|,
\end{align}
for every $ 0 \leq s \leq \gamma$. Now, in view of \eqref{eq:bound_real_Q_commutaror_lenght_2}, we obtain the following estimate
\begin{align*}
	\Bigg| \int_{[0, \gamma]} \bigg\langle \big(\chi\mathrm{L}^\sharp(u)\big)(x^\prime, t^\prime, s),  & e^{Q(x,t,x^\prime, t^\prime, \xi,\tau, s)}  \det Z_{x,t}(x^\prime,t^\prime, s) \bigg\rangle \mathrm{d} s \Bigg| \leq \\
	& \leq \int_{[0,\gamma]} C_k s^k \sum_{|\alpha|\leq m_1} \sup_{(x^\prime,t^\prime)\in V_2} \bigg| \partial_{x^\prime,t^\prime}^\alpha \chi(x^\prime,t^\prime) e^{Q(x,t,x^\prime, t^\prime, \xi,\tau, s)}  \det Z_{x,t}(x^\prime,t^\prime, s) \bigg| \mathrm{d}s \\
	& \leq \text{Const}(k,m_1) \cdot \int_0^\infty s^k |\xi|^{m_1} e^{-\frac{\rho {s}^2}{16}|\xi|} \mathrm{d} s \\
	& \leq \text{Const}(k,m_1) |\xi|^{\frac{-(k+1)}{2} + m_1},
\end{align*}
for every $k \in \mathbb{Z}_+$. Summing up, and noticing that $|\xi| \geq \frac{2}{3} |(\xi, \tau)|$ if $(\xi,\tau) \in \Gamma$, we have proved that for each $k \in \mathbb{Z}_+$ there exists $C^\sharp_k > 0$ such that
\begin{equation*}
	|\mathfrak{F}_\kappa[\chi u_0](x,t;\xi,\tau)| \leq \frac{C^\sharp_k}{|(\xi,\tau)|^k},
	\eqno{\forall (x,t)\in V_0, (\xi, \tau) \in \Gamma, |\xi| \geq 1, \tau > 0.}
\end{equation*}
As mentioned before, in order to obtain the same estimate for $\tau < 0$ it is enough to replace in \eqref{eq:Stokes1} the interval $[0, \gamma]$ by $[-\gamma, 0]$.
\end{proof}
Now let $r > 0$ and consider $\mathrm{A}(x,t)$ a $r \times r$ matrix with complex-valued, $\mathcal{C}^\infty$-smooth coefficients. Define
\begin{equation}
	\mathrm{D} = \mathrm{L} + \mathrm{A}.
\end{equation}
Now we can prove the main Theorem of this section:
\begin{thm}\label{thm:microlocal_nonhomogeneous_system}
	Let $V\subset U$ be an open neighborhood of the origin, and let $\sigma \in \mathcal{D}^\prime(V; \mathbb{C}^r)$ be a solution of $\mathrm{D}\sigma = 0$. Let $\xi^0 \in \mathbb{R}^N \setminus \{ 0 \}$, and assume that $\frac{\partial b}{\partial t}(0) \cdot \xi^0 > 0$. Then $(\xi^0, 0) \notin \mathrm{WF} (\sigma)$.
\end{thm}
\begin{proof}
	Let $\{ e^1, \dots, e^r\}$ the canonical basis for $\mathbb{C}^r$. We set $\mathrm{D}_1 = \frac{\partial}{\partial s} + i \mathrm{D}$.
Let $V \Subset U$ be an neighborhood of the origin, and let ${\omega}^1(x,t,s), \dots, {\omega}^r(x,t,s)$ be $\mathcal{C}^\infty$-smooth approximate solutions for $\mathrm{D}_1$ on $V$, such that $\omega^j(x,t,0) = e^j$, for $j = 1, \dots, r$. Let $\delta>0$ be small enough so 
the vectors  ${\omega}^1(x,t,s), \dots, {\omega}^r(x,t,s)$ span $\mathbb{C}^r$ on $V \times ]-\delta , \delta[$. Therefore we can write $e^j = M_\ell^j \omega^\ell $, in $V \times]-\delta,\delta[$, where $M_\ell^j$ is an invertible matrix with $\mathcal{C}^\infty$-coeficients. Let $\sigma \in \mathcal{D}^\prime(V;\mathbb{C}^r)$ be a solution of $D\sigma =0$, \textit{i.e.} $\sigma = \sigma_j e^j$, where $\sigma_j \in \mathcal{D}^\prime(V)$. So we can write $\sigma = \sigma_j M_k^j \omega^k$. Now since $D_1\sigma = 0$, we have that
\begin{align*}
	(\mathrm{L}_1 \sigma_j M_k^j) \omega^k & = -  \sigma_j M_k^j \mathrm{D}_1 \omega^k\\
	& = - \sigma_j M_k^j \mathrm{D}_\ell^k(\mathrm{L}_1) \omega^\ell,
\end{align*}
where $\mathrm{D}_1 \omega^k =  \mathrm{D}_\ell^k(\mathrm{L}_1) \omega^\ell$. Since $\{\omega^1, \dots, \omega^r\}$ is a frame in $V\times]-\delta,\delta[$, we have that
\begin{equation*}
	(\mathrm{L}_1 \sigma_j M_\ell^j) = -  \sigma_j M_k^j \mathrm{D}_\ell^k(\mathrm{L}_1),
\end{equation*}
and since $\omega^j$ is an $s$-approximate solution for $\mathrm{D}_1$, we have that $\mathrm{D}_\ell^k(\mathrm{L}_1)$ is flat at $s=0$, therefore $(\mathrm{L}_1 \sigma_j M_\ell^j) \in \mathcal{C}_\text{flat}^\infty(]-\delta, \delta[; \mathcal{D}^\prime(V))$, for every $\ell = 1, \dots,r$, and $\sigma_j M_\ell^j\big|_{s=0} = \sigma_\ell$. By Theorem \ref{thm:microlocal_approximate_solution} we have that $(\xi^0,0) \notin \mathrm{WF}(\sigma_\ell)$, for every $\ell=1,\dots,r$, thus $(\xi^0,0) \notin \mathrm{WF}(\sigma)$.		 
\end{proof}
%
%
%%%%%%%%%%%%%%%%%%%%%%%%%%%%%%%%
%%%%%%%%%%%%%%%%%%%%%%%%%%%%%%%%
%
\section{Proof of Theorem \ref{thm:levi_microlocal_regularity}}
%
%%%%%%%%%%%%%%%%%%%%%%%%%%%%%%%%
%%%%%%%%%%%%%%%%%%%%%%%%%%%%%%%%
%
Let $p \in \Omega$ and let $\sigma \in \mathrm{T}^0_p$ be such that the Levi form in $\sigma$ has one negative eigenvalue. Therefore there exist $U \subset \Omega$ an open neighborhood of $p$, and a section $\mathrm{L} \in \Gamma(U, \mathcal{V})$ such that $\sigma([\mathrm{L}, \overline{\mathrm{L}}]_p)/2i < 0$. Now since $[-i\mathrm{L}, i\overline{\mathrm{L}}] = [\mathrm{L}, \overline{\mathrm{L}}]$, we can assume, shrinking $U$ if necessary, that $\Re \mathrm{L} \neq 0$ on $U$. Rectifying the $\Re \mathrm{L}$, we can find coordinates on $U$, $(x_1, \dots, x_N, t)$, vanishing on $p$, such that in this coordinate $\mathrm{L}$ is written as
\begin{equation*}
	\mathrm{L} = \frac{\partial}{\partial t} + i b(x,t)\cdot \frac{\partial}{\partial x},
\end{equation*}
where the map $b(x,t) = (b_1(x,t), \dots, b_N(x,t))$ is real-valued and $\mathcal{C}^\infty$-smooth. Since $\sigma \in \mathrm{T}^0_p = \mathrm{T}^\prime_0\cap\mathbb{C}T_p^\ast\Omega$, we have that $\sigma (\mathrm{L}_p) = 0$, which implies that, in the coordinate system $(x,t)$, $ \sigma = (\xi^0, 0)$, for some $\xi^0 \in \mathbb{R}^N \setminus \{0\}$. So the condition $\sigma([\mathrm{L}, \overline{\mathrm{L}}]_p)/2i < 0$ translates to $\frac{\partial b(0)}{\partial t}\cdot \xi^0 > 0$. Now let $u$ be a distributional section solution on $U$ (here we allow $U$ to be as small as necessary). Let $\{f^1, \dots, f^r$\} be a frame of $E$ on $U$. In the same notation of section \ref{sec:V_bundles} we have that $u$ is a solution of
\begin{equation*}
	\big(\mathrm{L} + D(\mathrm{L})\big) u = 0, 
\end{equation*}
where $D(\mathrm{L})$ is the connection matrix of $D_\mathrm{L}$ with respect to the frame $\{f^1, \dots, f^r\}$. Therefore Theorem \ref{thm:levi_microlocal_regularity} follows from Theorem \ref{thm:microlocal_nonhomogeneous_system}.

\end{document}